\documentclass[12pt,a4paper,english]{article}

\usepackage[english]{babel}
\usepackage{graphicx}
\usepackage{amsfonts}
\usepackage{amsmath,amssymb}
\usepackage{subfig}
\usepackage{url}
\usepackage{amsthm}
\usepackage{bm}
\usepackage[round]{natbib}
\usepackage{authblk}

\usepackage{tabularx}
\usepackage{multirow}
\usepackage{color}
\usepackage{xcolor}
\RequirePackage[colorlinks,citecolor=blue,urlcolor=blue]{hyperref}
\usepackage{cleveref}

\providecommand{\keywords}[1]
{
  \small	
  \textbf{\textit{Keywords---}} #1
}
 
\newcommand{\Ds}{{\rm D}_s}

\newcommand{\vett}[1]{\boldsymbol{#1}}
\newcommand{\Real}{\mathbb{R}}

\newtheorem{thm}{Theorem}
\newtheorem{cor}{Corollary}

\newtheorem{lem}{Lemma}

\def\eqalign#1{\null\,\vcenter{\openup\jot
   \ialign{\strut\hfil$\displaystyle{##}$&$
   \displaystyle{{}##}$\hfil \crcr#1\crcr}}\,}

\newcommand{\E}{\mathbb{E}}

\newcommand{\red}[1]{{{#1}}}  

\begin{document}
	\title{Designing to detect heteroscedasticity in a regression model}
	
	\author[1*]{Lanteri, A.}
	\author[2]{Leorato, S.}
	\author[3]{L\'opez-Fidalgo, J.}
	\author[4]{Tommasi, C.}
	\affil[1,2,4]{Department of Economics, Management and Quantitative Methods, University of Milan, Italy.}
	\affil[3]{Institute of Data Science and Artificial Intelligence, University of Navarre, Spain}
	\affil[ ]{ }
	\affil[*]{Corresponding author contacts: alessandro.lanteri@unimi.it 

	Via Conservatorio, 7. 20122 Milano, Italy.  }
 \date{}
	\maketitle

	\begin{abstract}
We consider the problem of designing experiments to detect the presence of a specified heteroscedastity in a non-linear Gaussian regression model. 
In this framework, we focus on the $\Ds$- and KL-criteria and study their relationship with the noncentrality parameter of the asymptotic chi-squared distribution of a likelihood-based test, for local alternatives. 
Specifically, we found that when the variance function depends just on one parameter, the two criteria coincide asymptotically and in particular, the ${\rm D}_1$-criterion is proportional to the noncentrality parameter. 
Differently, if the variance function depends on a vector of parameters, then the KL-optimum design converges to the design that maximizes the noncentrality parameter. 
Furthermore, we confirm our theoretical findings through a simulation study concerning the computation of asymptotic and exact powers of the log-likelihood ratio statistic.
	\end{abstract}	
	
\section{Introduction}
In the literature of Optimal Design of Experiments, many papers concern  precise estimation of  the regression coefficients or some transformation of them; classical references are  \cite{Atkinson_etal2007}, \cite{Silvey1980} and \cite{Pazman1986}, among others. Much less attention is deserved to the estimation of the error variance function, because frequently the error term is assumed to be homoscedastic or with 
a heteroscedasticity derived from a variance function known up to a proportionality constant. 

	The goal of this paper is to design an experiment to detect a specific kind of heteroscedasticity in a non-linear Gaussian regression model, 
\begin{equation}
	y_i=\mu(\bm{x}_i;\bm{\beta})+\varepsilon_i,\; \varepsilon_i\sim N(0;\sigma^2 h(\bm{x}_i;\bm{\gamma})),\quad i=1,\ldots,n 
	\label{model}
\end{equation}
where $y_i$ is a response variable	with non-linear  mean function $\mu(\bm{x}_i;\bm{\beta})$ and  error variance $\sigma^2 h(\bm{x}_i;\bm{\gamma})$;  herein $\bm{\beta}\in \mathbb{R}^m$ is the vector of regression coefficients, $\sigma^2$ is an unknown constant and $h: \Real^p\times\Real^s\mapsto \Real_+$ is a continuous positive function that depends on  
a parameter vector $\bm{\gamma}\in \mathbb{R}^s$. Assume that there exists a value ${\bm{\gamma}}_0$ such that $h(\bm{x}_i;{\bm{\gamma}}_0)=1$; in other words,  ${\bm{\gamma}}={\bm{\gamma}}_0$ leads to the  homoscedastic model.
	
	In many practical problems, it may be convenient to test for the model heteroscedasticity with a likelihood-based test (log-likelihood ratio, score or Wald test), since  its asymptotic distribution is known to be a chi-squared random variable with $s$ degrees of freedom.  
	In particular, we consider local alternatives: 
	\begin{equation}
	\left\{ 
	\begin{array}{cl}
	H_0: & \bm{\gamma}={\bm{\gamma}}_0 \\
	H_1: & \bm{\gamma}={\bm{\gamma}}_0+\frac{\bm{\lambda}}{\sqrt{n}},\; \bm{\lambda}\neq\bm{0},  \\
	\end{array}	
		\right.
		\label{system}
	\end{equation}  
 with $\bm{\lambda}\in \mathbb{R}^s$, and we aim at designing an experiment with the goal of maximizing in some sense the (asymptotic) power  of a likelihood-based test. 
In this study, we justify the use of the $\Ds$- and the KL-criterion \citep{Fidalgo_etal2007} to optimally design an experiment with the inferential goal of testing hypotheses (\ref{system}). Some other authors who considered the inferential issue of hypotheses testing (but concerning the regression functions) are  \cite{STIGLER1971}, \cite{SPRUILL1990}, \cite{Dette_Titoff2009}, who used the $\Ds$-optimality and/or the $T$-criterion proposed by \cite{Atkinson_Fedorov1975}; see also Chapter 6 of  \cite{Fedorov1972}. 

Both the $\Ds$- and KL- criteria are related, in a different way, to the noncentrality parameter of the asymptotic chi-squared distribution of a likelihood-based test. We prove that the first order approximation of KL-criterion for discriminating between the homoscedastic and the heteroscedastic models coincides with the noncentrality parameter. 
Although the same result does not hold in general for the $\Ds$-optimality, the $\Ds$-criterion  is proportional to the noncentrality parameter when $\gamma$ is scalar.
Therefore, in the scalar case, the two criteria lead to the same optimal design asymptotically.

\red{The most relevant result is that the sequence of the KL-optimum designs are proved to converge to the design that maximizes the noncentrality parameter. Furthermore, the analytic expression of this limiting design (at which we expect to reach the maximum asymptotic power of the test) is also provided. } 

The paper is organized as follows. In Section \ref{sec:2}, with regard to problem  (\ref{system}), we give the expression of the noncentrality parameter for the chi-square  asymptotic distribution of a likelihood-based test. We also recall how the $\Ds$-criterion is related to the noncentralitity parameter and  we prove that, in this context, the $\Ds$-optimum design can be obtained as a $\rm D$-optimality solution.  
In Section \ref{sec:3}, we provide the closed-form  expression for the KL-optimal design, which is quite uncommon. \red{Furthermore, we prove that  the KL-criterion differs from the noncentrality parameter by a term that goes to zero uniformly; this implies that the sequence of the KL-optimum designs converges to the design maximizing the noncentrality parameter. } Section \ref{examples} presents three examples where we compute  the asymptotic power of the log-likelihood ratio statistic for several designs; in addition, the exact powers are computed through a simulation study. \red{These numerical outcomes confirm the theoretical findings.} Finally, in the Appendix we collect the proofs of all the theoretical results stated in the main body of this manuscript.
\section{$\Ds$-criterion and noncentrality parameter}
\label{sec:2}

Let $\bm{x}$ denote an experimental point that can vary in a compact experimental domain ${\cal X}\subseteq {\rm I\! R}^p$. Following \cite{Kiefer_Wolfowitz1959}, an approximate design
$\xi=\left\{
\begin{array}{ccc}
\bm{x}_1 & \ldots & \bm{x}_k\\
\xi(\bm{x}_1)& \ldots & \xi(\bm{x}_k)
\end{array}
\right\}
$  
is a discrete probability measure on ${\cal X}$.  \red{Let $\Xi$ denote the set of all possible designs.} A design is said \lq\lq optimal" if it maximizes some concave criterion function of $\xi\in \Xi$, which reflects the inferential goal.
The $\Ds$-criterion of optimality is commonly used for precise estimation of a subset of parameters. Herein, we are interested in estimating as precisely as possible the variance parameter vector $\bm{\gamma}$. The $\Ds$-optimum design maximizes the precision matrix of the maximum likelihood estimator of $\bm{\gamma}$.

From model (\ref{model}), the log-likelihood function for one observation at $\bm{x}$ is
\begin{equation*}
l(y; \bm{\beta},\sigma^2,\bm{\gamma},\bm{x}) = \left\{ -\frac{1}{2}\frac{[y-\mu(\bm{x};\bm{\beta})]^2}{\sigma^2 h(\bm{x};\bm{\gamma})} -\frac{1}{2}\left[\log(2\pi)+\log(\sigma^2)+\log(h(\bm{x};\bm{\gamma})) \right] \right\}.
\end{equation*}
The score vector, denoted by $\bm{u}=(\bm{u}_{\bm{\beta}}^T, u_{\sigma^2},  \bm{u}_{\bm{\gamma}}^T)^T$ is
\begin{align*}
\bm{u}_{\bm{\beta}} &= \frac{\partial l(y; \bm{\beta},\sigma^2,\bm{\gamma},\bm{x})}{\partial \bm{\beta}}
=
\frac{[y-\mu(\bm{x};\bm{\beta})]}{\sigma^2 h(\bm{x};\bm{\gamma})}\,
\nabla \mu(\bm{x};\bm{\beta}) \\
u_{\sigma^2} &= \frac{\partial l(y; \bm{\beta},\sigma^2,\bm{\gamma},\bm{x})}{\partial \sigma^2}= \frac{1}{2} \frac{[y-\mu(\bm{x};\bm{\beta})]^2}{\sigma^4\, h(\bm{x};\bm{\gamma})}-\frac{1}{2\sigma^2}  \\
\bm{u}_{\bm{\gamma}} &= \frac{\partial l(y; \bm{\beta},\sigma^2,\bm{\gamma},\bm{x})}{\partial \bm{\gamma}}
= \frac{1}{2} \frac{\nabla h(\bm{x};\bm{\gamma}) \, [y-\mu(\bm{x};\bm{\beta})]^2}{\sigma^2 h(\bm{x};\bm{\gamma})^2}-\frac{1}{2}\frac{\nabla h(\bm{x};\bm{\gamma})}{h(\bm{x};\bm{\gamma})} 
\end{align*}
where  
$\displaystyle\nabla \mu(\bm{x};\bm{\beta})=\frac{\partial\mu(\bm{x};\bm{\beta})}{\partial\bm{\beta} }$ and $\displaystyle\nabla h(\bm{x};\bm{\gamma})=\frac{\partial h(\bm{x};\bm{\gamma})}{\partial \bm{\gamma} }$.
After some algebra, we have that 
\begin{eqnarray*}
	\bm{J}(\bm{x};\bm{\beta}, \sigma^2,\bm{\gamma})&=&
	E[\bm{u}\bm{u}^T] \\
	&=& \begin{bmatrix} 
		\displaystyle \frac{\nabla \mu(\bm{x};\bm{\beta}) 
			\nabla \mu(\bm{x};\bm{\beta})^T}{\sigma^2\,h(\bm{x};\bm{\gamma})} & 0 & \bm{0}^T \\[.2cm]
		0 & \displaystyle\frac{1}{2\,\sigma^4} & \displaystyle\frac{\nabla h(\bm{x};\bm{\gamma})^T}{2\,\sigma^2 h(\bm{x};\bm{\gamma})} \\[.3cm]
		\bm{0} & \displaystyle\frac{\nabla h(\bm{x};\bm{\gamma})}{2\,\sigma^2 h(\bm{x};\bm{\gamma})}  & \displaystyle\frac{\nabla h(\bm{x};\bm{\gamma})  
			\nabla h(\bm{x};\bm{\gamma})^T}{2\,h(\bm{x};\bm{\gamma})^2}  
	\end{bmatrix}
\end{eqnarray*}  
where the expectation is taken with respect to the Normal law.

The Fisher Information matrix is 
$$
\mathcal{I}(\xi;\bm{\beta},\sigma^2, \bm{\gamma}) = \sum_{i=1}^k \bm{J}(\bm{x}_i;\bm{\beta}, \sigma^2,\bm{\gamma})\, \xi(\bm{x}_i)=
\begin{bmatrix}
\mathcal{I}_{11} & \mathcal{I}_{12} \\
\mathcal{I}_{12}^T & \mathcal{I}_{22}
\end{bmatrix}
$$	
where
$$
\mathcal{I}_{11}\!=\!\begin{bmatrix}
\bm{M}(\xi;\!\bm{\beta},\!\sigma^2,\! \bm{\gamma}) & 0  \\
0 & \displaystyle\frac{1}{2\,\sigma^4} 
\end{bmatrix}\!\!; \;\;
\bm{M}(\xi;\!\bm{\beta},\!\sigma^2,\! \bm{\gamma})\!=\!
\sum_{i=1}^k \! \displaystyle \frac{\nabla \mu(\bm{x}_i;\bm{\beta}) 
	\nabla \mu(\bm{x}_i;\bm{\beta})^T}{\sigma^2\,h(\bm{x}_i;\bm{\gamma})} \,\xi(\bm{x}_i)
$$
$$
\mathcal{I}_{12}=\begin{bmatrix}
\bm{0}^T \\
\displaystyle \sum_{i=1}^k\frac{\nabla h(\bm{x}_i;\bm{\gamma})^T}{2\,\sigma^2 h(\bm{x}_i;\bm{\gamma})} \, \xi(\bm{x}_i)
\end{bmatrix} 
;\quad
\mathcal{I}_{22}=
\displaystyle\sum_{i=1}^k 
\frac{\nabla h(\bm{x}_i;\bm{\gamma}) \nabla h(\bm{x}_i;\bm{\gamma})^T}{2\,h(\bm{x}_i;\bm{\gamma})^2}\, \xi(\bm{x}_i).
$$
The asymptotic covariance matrix of the MLE for $\bm{\gamma}$ is $[\mathcal{I}_{22.1}(\xi;\bm{\gamma})]^{-1}$, where 
$\mathcal{I}_{22.1}(\xi;\bm{\gamma})=
\mathcal{I}_{22} -\mathcal{I}_{12}^T \mathcal{I}_{11}^{-1} \mathcal{I}_{12}$ is the Schur complement of $\mathcal{I}_{22}$ in $\mathcal{I}(\xi;\bm{\beta},\sigma^2, \bm{\gamma})$.
The $\Ds$-optimum design for $\bm{\gamma}$ minimizes the generalized variance of the MLE for $\bm{\gamma}$:
\begin{equation}
\xi_{\Ds}=\arg\max_\xi |\mathcal{I}_{22.1}(\xi;\bm{\gamma})|.
\label{Ds_schur}
\end{equation}
The following theorem proves that the $\Ds$-optimality for $\bm{\gamma}$ actually is $\rm D$-optimality for an \lq\lq auxiliary'' linear regression. This result is quite useful from a practical point of view, because there is much statistical software to compute $\rm D$-optimum designs, which can be exploited to determine $\xi_{\Ds}$.
\begin{thm}
	\label{th:Ds_to_D_aux}
	The $\Ds$-optimum design for $\bm{\gamma}$ coincides with the $\rm D$-optimum design for the \lq\lq auxiliary'' linear regression model:
	$$y_i=\alpha_0+\bm{\alpha}^T \nabla \log h(\bm{x}_i;\tilde{\bm{\gamma}})+\varepsilon_i, \; \varepsilon_i\sim N(0;\sigma^2 ), \quad i=1,\ldots,n,
	$$ 
	where $\nabla \log h(\bm{x};\bm{\gamma})=\left(\frac{\partial \log h(\bm{x};\bm{\gamma})}{\partial \gamma_1},\ldots, \frac{\partial \log h(\bm{x};\bm{\gamma})}{\partial \gamma_s} \right)^T$ and $\tilde{\bm{\gamma}}$ is a nominal value for $\bm{\gamma}$.
\end{thm}
\begin{cor}
	\label{cor:1}
	If $\gamma$ is scalar, then $s=1$ and  
	the ${\rm D}_1$-optimal design for $\gamma$ is 
	\begin{equation}
	\xi_{\rm{D}_1}=\left\{
	\begin{array}{cc}
	{x}_l  & {x}_u\\
	0.5 & 0.5
	\end{array}
	\right\}
	\label{D1}
	\end{equation}
	where ${x}_l\in \{x\!: \operatorname*{argmin}_x \nabla \log h(x,\tilde{\gamma})\}$ and ${x}_u\in \{x\!: \operatorname*{argmax}_x \nabla \log h(x,\tilde{\gamma})\}$ with $\nabla \log h(x,\gamma)=\frac{\partial \log h(x_i;\gamma)}{\partial \gamma}$.
\end{cor}
\medskip
When the inferential goal is testing (\ref{system}), the $\Ds$-optimum design is commonly applied  because somehow it maximizes  the asymptotic power of a likelihood-based test.
It is well known that, under the alternative hypotheses,  the log-likelihood ratio, Wald and Score tests are asymptotically distributed as a chi-squared random variable with noncentrality parameter $\zeta(\xi;\bm{\lambda};\bm{\gamma}_0)=\bm{\lambda}^T \mathcal{I}_{22.1}(\xi;{\bm{\gamma}}_0) \bm{\lambda}$. 
Therefore, the  $\Ds$-optimum design (\ref{Ds_schur}) with $\bm{\gamma}=\bm{\gamma}_0$  maximizes \lq\lq in some sense'' $\zeta(\xi;\bm{\lambda};\bm{\gamma}_0)$, without considering any specific value of $\bm{\lambda}$.

Replacing $\mathcal{I}_{11}$, $\mathcal{I}_{12}$ and $\mathcal{I}_{22}$ into the expressions of $\mathcal{I}_{22.1}(\xi;{\bm{\gamma}})$ and setting $\bm{\gamma}={\bm{\gamma}}_0$, we can  provide the analytic expression of the noncentrality parameter.
\begin{thm}
	\label{th:1}
	Given model (\ref{model}), the nononcentrality parameter of the asymptotic distribution of a likelihood statistic for testing (\ref{system}) is 
	\begin{eqnarray}
	\zeta(\xi;\bm{\lambda};\bm{\gamma}_0)
	&=&
	\dfrac{1}{2}\,\bm{\lambda}^T \left(\!\sum_{i=1}^k{\!\nabla h(\bm{x}_i;{\bm{\gamma}}_0)} \nabla h(\bm{x}_i;{\bm{\gamma}}_0)^T \xi(\bm{x}_i)\right.\nonumber
	\\
	&-&
	\left.
	\sum_{i=1}^k\! \nabla h(\bm{x}_i;{\bm{\gamma}}_0) \xi(\bm{x}_i)\, \sum_{i=1}^k \!\nabla h(\bm{x}_i;{\bm{\gamma}}_0)^T \xi(\bm{x}_i)\!\right) \bm{\lambda}.
	\label{noncentrality_parameter}
	\end{eqnarray}
\end{thm}
\medskip	
\noindent
{\bf Remark 1.} From  Corollary \ref{cor:1}, if $\gamma$ is scalar and the inferential goal is testing hypothesis (\ref{system}), then the ${\rm D}_1$-optimal design for $\gamma$ is equally supported at ${x}_l\in \operatorname*{argmin}_x \nabla h(x;\gamma_0)$ and ${x}_u\in  \operatorname*{argmax}_x \nabla h(x;\gamma_0)$. Furthermore, from Theorem \ref{th:1}, it  exactly maximizes the noncentrality parameter. 
\section{Noncentrality parameter as first order approximation of the KL-criterion}
\label{sec:3}
The KL-criterion proposed by \cite{Fidalgo_etal2007} can be applied	to discriminate between heteroscedastic and homoscedastic models, under $H_1$ and $H_0$, respectively. If the heteroscedastic model 
	is completely known, i.e. if 
 	 $(\bm{\beta}_1^T,\sigma_1,\bm{\gamma}_1^T)$ are the assumed known parameter values under $H_1$,  then the KL-criterion is:
\begin{equation}
    \label{I12}
	I_{12}(\xi;\bm{\gamma}_1)=\min_{\bm{\beta},\sigma^2}\sum_{i=1}^k 
	\left(\frac{\sigma_1^2 h(\bm{x}_i)+\left[\mu(\bm{x}_i;\bm{\beta}_1)-\mu(\bm{x}_i;\bm{\beta})\right]^2}{\sigma^2} -\log \frac{\sigma_1^2 h(\bm{x}_i)}{\sigma^2} \right) \;\xi(\bm{x}_i), 
\end{equation}
where $h(\bm{x}_i)=h(\bm{x}_i; \bm{\gamma}_1)$.
\begin{thm}
	\label{thm:ex_lem1}
	Given a design $\xi$, let $A_h\!=\!\sum_{i=1}^k h(\bm{x}_i) \,\xi(\bm{x}_i)$ and $G_h\!=\!\prod_{i=1}^{k}[h(\bm{x}_i)]^{\xi(\bm{x}_i)}$ be the arithmetic and the geometric means of the values $h(\bm{x}_i)$, $i=1,\ldots,k$, respectively.	
	The KL-criterion becomes:
	\begin{equation}
	\label{eqn:I21}
	I_{12}(\xi;\bm{\gamma}_1) =
		1+\log A_h- \log G_h.
	\end{equation}
\end{thm}
For the next results we need to introduce some regularity assumptions on $h$.
Let $\mathcal{B}(\bm{\gamma}_0,r)$ be the ball of radius $r$ centered at $\bm{\gamma}_0$.
We assume that there exists a value $r>0$, and finite constants $M_1$ and $M_2$ such that the function $h(\bm{x},\bm{\gamma})$ is twice continuously differentiable with respect to $\bm{\gamma} \in \mathcal{B}(\bm{\gamma}_0,r)$ for all $\bm{x} \in \cal{X}$ and that:
\begin{equation}\label{conditions_unif_continuity}
\begin{array}{c}
\sup_{x \in \mathcal{X}} \|\nabla h(\bm{x},\bm{\gamma_0})\| \leq M_1 < \infty; \\
\\
\sup_{x \in \mathcal{X}} \sup_{\bm{\gamma} \in \mathcal{B}(\bm{\gamma}_0,r)} \| h''(\bm{x},\bm{\bm{\gamma}})\|_2 \leq M_2 <\infty,
\end{array}
\end{equation}
where, $h''(\bm{x},\bm{\gamma})=\frac{\partial^2 }{\partial \bm{\gamma}\partial \bm{\gamma}^T}h(\bm{x},\bm{\gamma})$ and $\|\bm{A}\|_2$ indicates the spectral norm of a square matrix $\bm{A}$: $\|\bm{A}\|_2=\sup_{u \not=0} \frac{\|\bm{Au}\|}{\|\bm{u}\|} $, that coincides with the square root of the largest eigenvalue of $\bm{A}^T\bm{A}$.
From Theorem \ref{thm:ex_lem1}, the next result follows, whose proof is given in the Appendix.

\begin{thm}
\label{th:3}
Let $\bm{\gamma}_1=\bm{\gamma}_0+\bm{\lambda}/\sqrt{n}$, and let assumption (\ref{conditions_unif_continuity}) hold for some constants $M_1$, $M_2$ and $r>0$. Then, the KL-criterion is related to the noncentrality parameter of the asymptotic distribution of a likelihood-based test by:
\begin{eqnarray*}
I_{12}\big(\xi;\bm{\gamma}_1\big) 
&=& 1
+\frac{1}{n} \, \zeta(\xi;\bm{\lambda};\bm{\gamma}_0) + O\left(\dfrac{||\bm{\lambda}||^3}{n^{\frac{3}{2}}}\right).
\end{eqnarray*}
\end{thm}
\red{Theorem \ref{th:3} essentially states that if  $n$ is large enough to make the error term negligible, the KL-optimum design maximizes the noncentrality parameter.
Therefore, when $\gamma$ is scalar (and $n$ is suffiently large),  KL- and ${\rm D}_1$-optimal designs should almost coincide, because the ${\rm D}_1$-optimality is proportional to the noncentrality parameter. Differently, when $\bm{\gamma}$ is not scalar, the $\Ds$-criterion is no longer proportional to the noncentrality parameter and therefore, KL- and $\Ds$-optimum designs differ. }
In Section \ref{examples}, for some values of $\bm{\lambda}$ and $n$, we compare the power of the log-likelihood ratio test when different designs are applied: the KL-optimum design for several specific values of $\bm{\lambda}$, the $\Ds$-optimal design and a uniform design, which is frequently applied in practice. 

The following theorem has been very helpful to develop the simulation study, because it provides a closed form for the KL-optimum design. 

\begin{thm}
	\label{th:KL}
	Let us denote by $\underline{h}=\inf_{\vett x} h(\vett x; \bm{\gamma}_1)>0 $ and $\overline{h}=\sup_{\vett x} h(\vett x;\bm{\gamma}_1)<\infty $. Let further $\mathcal{X}_l=\{\vett x:\,h(\vett x;\bm{\gamma}_1)=\underline{h} \}$ and $\mathcal{X}_u=\{\vett x:\, h(\vett x;\bm{\gamma}_1)=\overline{h} \}$. 
	A KL-optimal design is
	\begin{equation}
	\label{KL-design}
	\xi_{\bm{\gamma}_1}^{KL}=\left\{
	\begin{array}{cc}
	\underline{\bm{x}}  & \overline{\bm{x}}\\
	\omega & 1-\omega
	\end{array}
	\right\},\quad \omega:=\omega(\bm{\gamma}_1)=\frac{\overline{h}}{\overline{h}-\underline{h}}-\frac{1}{\log \overline{h}-\log\underline{h}}
\end{equation}
	where $\underline{\bm{x}}\in \mathcal{X}_l$ and $\overline{\bm{x}}\in \mathcal{X}_u$. \\
	If $n \,\omega $ is not an integer number, then the best approximation is its integer part. 
	
\end{thm}
\red{Under mild assumptions, the following theorem proves that,  when $n$ goes to infinity (and thus $\bm{\gamma}_1 \rightarrow \bm{\gamma}_0$), the KL-optimum design tends to become equally supported at $\underline{\bm{x}}$ and $\overline{\bm{x}}$.  }
\begin{thm}\label{th:limitKL}
	\label{KL_asimpt}
	Let  $\overline{h}=\sup_{\bm{x}} h(\bm{x};\bm{\gamma}_1)$ and $\underline{h}=\inf_{\bm{x}} h(\bm{x};\bm{\gamma}_1)$.
	If $\underline{\bm{x}}=\arg\inf_{\bm{x}} h(\bm{x};\bm{\gamma}_1)$ and  $\overline{\bm{x}}=\arg\sup_{\bm{x}} h(\bm{x};\bm{\gamma}_1)$ do not depend on $\bm{\gamma}_1$ and $h(x;\bm{\gamma}_1)$ is such that $\overline{h}/\underline{h}\to 1$ as $\bm{\gamma}_1 \to \bm{\gamma}_0$, then
	$$\lim_{\bm{\gamma}_1 \to \bm{\gamma}_0} \omega(\bm{\gamma}_1)=1/2,$$
	where $\omega(\bm{\gamma}_1)$ is defined by (\ref{KL-design}).
	In particular,
	for all $\bm{\gamma}_1$ such that $\overline{h}/\underline{h}> 1$, $$\omega(\bm{\gamma}_1)> 1/2.$$
	
\end{thm}	
\red{Let the limiting KL-optimal design be denoted by $$\xi_{\gamma_0}^{KL}=\left\{
	\begin{array}{cc}
	\underline{\bm{x}}  & \overline{\bm{x}}\\
	0.5 & 0.5
	\end{array}
	\right\}.$$}
Theorem \ref{th:3} shows the asymptotic link between the KL-divergence and the noncentrality parameter of a likelihood-based test (under the alternative). 
\red{The following result proves that the asymptotic expansion in Theorem \ref{th:limitKL} holds uniformly in $\xi$, and this   implies that, as $n$ goes to infinity, the sequence of KL-optimal designs converges to the design which maximizes the noncentrality parameter. }

\begin{thm}\label{th:unif_cont}
Assume that (\ref{conditions_unif_continuity}) holds for some constants $M_1$, $M_2$ and $r>0$. Also assume that $\bm{\lambda}$ belongs to a bounded set $\bm{\Lambda}$: $\sup_{\bm{\lambda} \in\bm{\Lambda}} \|\bm{\lambda}\| \leq L<\infty$. 
Then the following result holds for all $\bm{\lambda} \in \bm{\Lambda}$
$$\sup_{\xi\in \Xi} \left| n( I_{12}(\xi;\bm{\gamma}_1)-1) - \zeta(\xi,\bm{\lambda},\bm{\gamma}_0) \right|  = O(n^{-1/2}),$$ 
\end{thm}

 \red{From Theorems \ref{th:limitKL} and \ref{th:unif_cont} we have the following result.
 \begin{cor} 
 	\label{Corollary:2}
 	The design which maximizes the noncentrality parameter $\zeta(\xi,\bm{\lambda},\bm{\gamma}_0) $ is the limiting KL-optimal design $\xi_{\gamma_0}^{KL}$.
\end{cor}
}

\section{Applications}
\label{examples}
In this Section we compare the performance of different designs in testing hypotheses \eqref{system} through a simulation study. We apply the log-likelihood ratio test using samples of data generated from each design and then we compare the corresponding powers. We also develop a comparison among the asymptotic powers of the test statistic.   
 
In these applications, we consider 
model (\ref{model}), where $\sigma^2=1$, the mean function is $\mu(x,\bm{\beta})=\beta_0 +\beta_1 x $, with $\beta_0=\beta_1=1$ while 
three different kinds of heteroscedasticity are examined:
\begin{description}
	\item[Case 1:] 
	$
	h(x,\gamma)=e^{\gamma x}
	$;
	\item[Case 2:]
	$
	h(x,\gamma)= 1 + \gamma x + \sin (2\pi \gamma x)
	$;
	\item[Case 3:]
	$
	h(x,\bm{\gamma})=1+\gamma_1 x+ \gamma_2 x^2.
	$
\end{description}
Without loss of generality, we consider $x \in  [0,1]$ to guarantee the above variance functions to be positive for $\vett \gamma \in \Real_+^s$.
In Cases 1 and 2, function $h(x,\gamma)$ depends on a one-dimensional parameter $\gamma \in \mathbb{R}$, and thus $s=1$. As we can appreciate from Figure \ref{fig:h1_h2}, in the first case $h(x,\gamma)$ is monotonic in $x$, while in the second, for  large enough values of $\gamma$, describes a non-monotonic heteroscedasticity. Case 3, instead, concerns a two-dimensional parameter $\bm{\gamma} \in \mathbb{R}^2$, and thus $s=2$. Note that, in all the three cases, the homoscedastic model is obtained when ${\gamma}_0={0}$  or $\bm{\gamma}_0=\bm{0}$ (Cases 1 and 2 or  Case 3, respectively). Therefore, from $H_1$ in (\ref{system}), $\gamma=\lambda/\sqrt{n}$ in the first two cases and $\bm{\gamma}=(\gamma_1,\gamma_2)^T$ in Case 3, where $\gamma_1=\lambda_1/\sqrt{n}$ and $\gamma_2=\lambda_2/\sqrt{n}$.

\begin{figure}
\includegraphics[width = 0.49\linewidth]{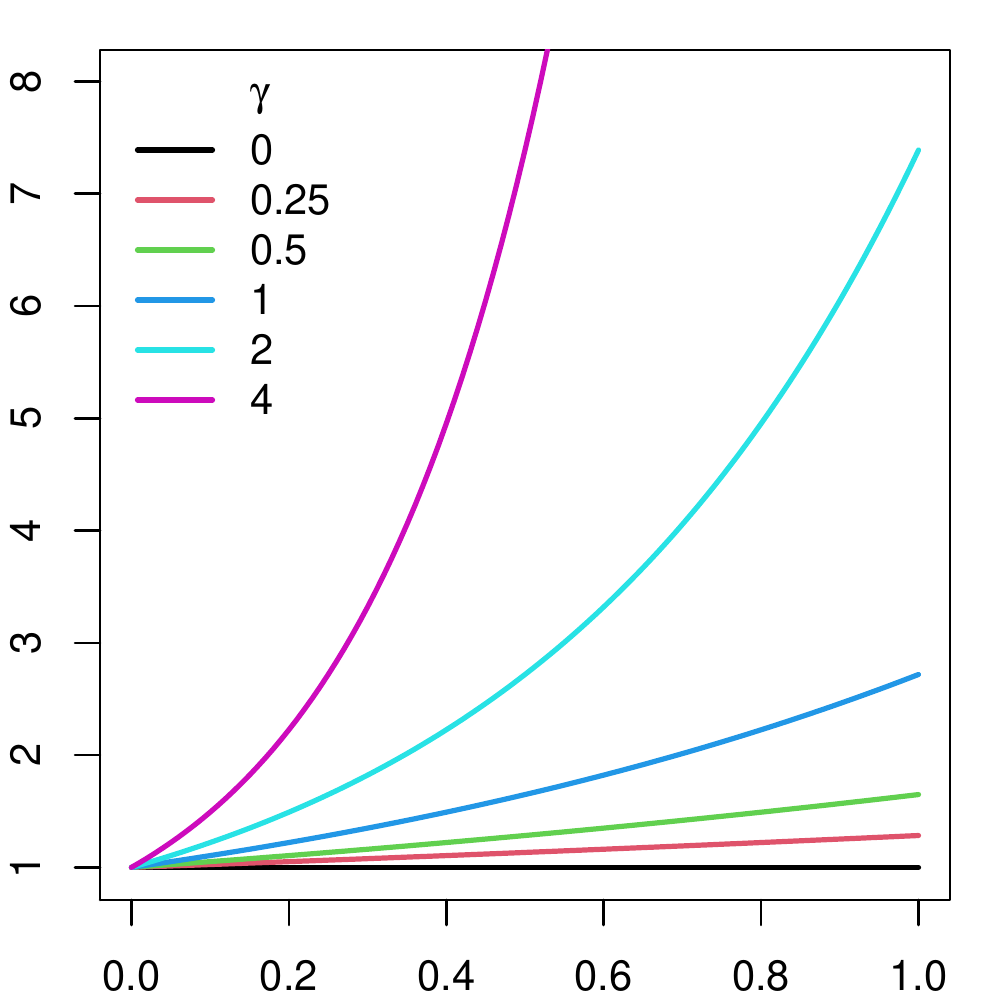}
\includegraphics[width = 0.49\linewidth]{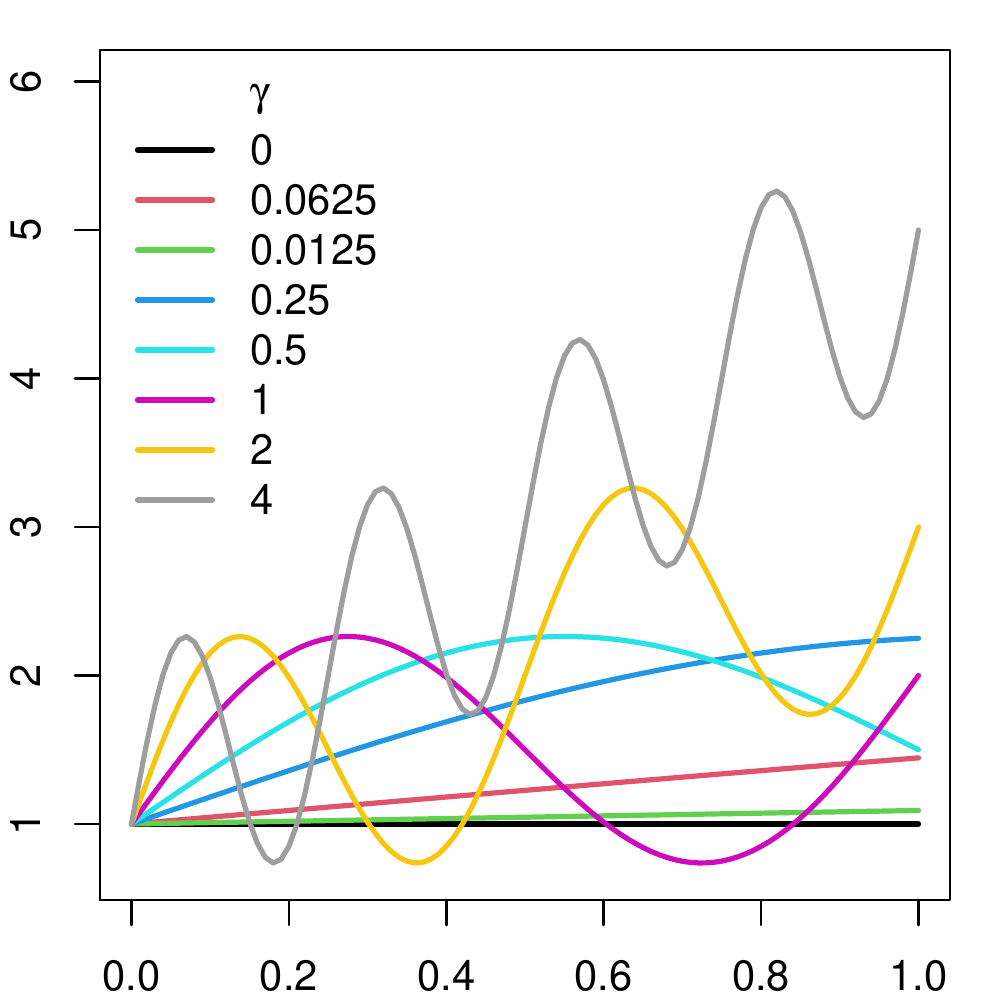}
\caption{Case 1 (left) and Case 2 (right) variance functions $h(x,\gamma)$, for different values of $\gamma$.}
\label{fig:h1_h2}
\end{figure}

 To discriminate these heteroscedastic models from the homoscedastic case, we firstly compute several KL-optimal designs (for different choices of $\lambda$, $\lambda_1$ and $\lambda_2$) and the $\Ds$-optimal design (for ${\gamma}={0}$ and $\bm{\gamma}=\bm{0}$). In addition, as a benchmark, we consider a uniform design 
$$\xi_{U_5}= \left\{ \begin{array}{ccccc}
0 & 0.25 & 0.5 & 0.75  & 1 \\ 
0.2 & 0.2 & 0.2 & 0.2 & 0.2 \\
\end{array}  \right\}.$$
Uniform designs are a common choice in practice and, in specific contexts, they provide nice theoretical guarantees (see \cite{Wiens1991}, \cite{Wiens2019} and \cite{Biedermann_Dette2001}).

 From each heteroscedastic model (Cases 1,2 and 3) we generate $n$ responses from model (\ref{model}), in different settings, applying the KL-optimal design, the $\Ds$-optimal design and the uniform design described above.
\begin{table}[h]
\centering
\resizebox{.99\textwidth}{!}{%
\begin{tabular}{c}
\begin{tabular}{rrr}
\multicolumn{3}{c}{Case 1} \\
\\[-.4cm]
$\xi_{{\rm D}_1} = \left\{ \begin{array}{cc}
0 & 1 \\
0.500 &
0.500 
\end{array}  \right\}$ & 
$\xi^{KL}_{0.25} = \left\{ \begin{array}{cc}
0 & 1 \\
0.521 &
0.479 
\end{array}  \right\}$ &
$\xi^{KL}_{0.5} = \left\{ \begin{array}{cc}
0 & 1 \\
0.542 &
0.458
\end{array} \right\}$ \\[.4cm] 
$\xi^{KL}_{1} = \left\{ \begin{array}{cc}
 0 &
1 \\
0.582 &
0.418
\end{array} \right\}$ & 
$\xi^{KL}_{2} = \left\{ \begin{array}{cc}
 0 &
1 \\
0.657 &
0.343
\end{array} \right\}$ &
$\xi^{KL}_{4} = \left\{ \begin{array}{cc}
0 & 1 \\
0.769 &
0.231
\end{array} \right\}$ 
\end{tabular} \\
$ $ \\
\begin{tabular}{rrrr}
\multicolumn{4}{c}{Case 2} \\
\\[-.4cm]

$\xi_{{\rm D}_1} = \left\{ \begin{array}{cc}
0 & 1 \\
0.500 & 0.500
\end{array} \right\}$ & 
$\xi^{KL}_{0.0625} = \left\{ \begin{array}{cc}
0 & 1 \\
0.504 & 0.496
\end{array} \right\}$& 
$\xi^{KL}_{0.125} = \left\{ \begin{array}{cc}
0 & 1 \\
0.507 & 0.493
\end{array} \right\}$  &
$\xi^{KL}_{0.25} = \left\{ \begin{array}{cc}
0 & 1 \\
0.510 & 0.490
\end{array} \right\}$ \\[.4cm] 
$\xi^{KL}_{0.5} = \left\{ \begin{array}{cc}
0.00 & 0.55\\
0.510 & 0.490
\end{array} \right\}$ & 
$\xi^{KL}_{1} = \left\{ \begin{array}{cc}
0.72 & 0.28 \\
0.512 & 0.488
\end{array} \right\}$ &
$\xi^{KL}_{2} = \left\{ \begin{array}{cc}
0.36 & 0.64 \\
0.520 & 0.480
\end{array} \right\}$ &
$\xi^{KL}_{4} = \left\{ \begin{array}{cc}
0.18 & 0.89 \\
0.532 & 0.468
\end{array} \right\}$ \\
\end{tabular}\\
$ $ \\
\begin{tabular}{rrr}
\multicolumn{3}{c}{Case 3} \\
\\[-.4cm]
$\xi_{{\rm D}_2} = \left\{ \begin{array}{ccc} 
				0 & 0.5 & 1 \\ 
				0.333 & 0.333 & 0.333 \\ 
				\end{array} \right\}$ &
$\xi^{KL}_{(0,0) }= \left\{ \begin{array}{cc} 
				0 & 1 \\ 
				0.5000 & 0.5000 \\
				\end{array} \right\}$	&				
$\xi^{KL}_{(0.05,0.05) }= \left\{ \begin{array}{cc} 
				0 & 1 \\ 
				0.5008 & 0.4992 \\
				\end{array} \right\}$ \\[.4cm]  
$\xi^{KL}_{(0.125,0.125)}  = \left\{ \begin{array}{cc} 
				0 & 1 \\ 
				0.5186 & 0.4814  \\
				\end{array} \right\}$ &
$\xi^{KL}_{(0.125,0.25)}  = \left\{ \begin{array}{cc} 
				0 & 1 \\ 
				0.5265 & 0.4735 \\
				\end{array} \right\}$ &
$\xi^{KL}_{(0.25,0.25)}  = \left\{ \begin{array}{cc} 
				0 & 1 \\ 
				0.5337 & 0.4663 \\ 
				\end{array} \right\}$ \\[.4cm] 
$\xi^{KL}_{(0.125,0.5)}  = \left\{ \begin{array}{cc} 
				0 & 1 \\ 
				0.5403 & 0.4597 \\
				\end{array} \right\}$ &
$\xi^{KL}_{(0.25,0.5)}  = \left\{ \begin{array}{cc} 
				0 & 1 \\ 
				0.5464 & 0.4536 \\
				\end{array} \right\}$ &				
$\xi^{KL}_{(0.5,0.5)}  = \left\{ \begin{array}{cc} 
				0 & 1 \\ 
				0.5573 & 0.4427 \\ 
				\end{array} \right\}$ \\[.4cm] 
$\xi^{KL}_{(0.25,1)}  = \left\{ \begin{array}{cc} 
				0 & 1 \\ 
				0.5668 & 0.4332 \\
				\end{array} \right\}$ &
$\xi^{KL}_{(0.5,1)}  = \left\{ \begin{array}{cc} 
				0 & 1 \\ 
				0.5753 & 0.4247 \\
				\end{array} \right\}$  &
$\xi^{KL}_{(1,1)}  = \left\{ \begin{array}{cc} 
				0 & 1 \\ 
				0.5898 & 0.4102 \\
				\end{array} \right\}$  \\[.4cm] 
$\xi^{KL}_{(0.5,2)}  = \left\{ \begin{array}{cc} 
				0 & 1 \\ 
				0.6018 & 0.3982 \\
				\end{array} \right\}$  &
$\xi^{KL}_{(1,2)}  = \left\{ \begin{array}{cc} 
				0 & 1 \\ 
				0.6120 & 0.3880 \\
				\end{array} \right\}$ &
$\xi^{KL}_{(2,2)}  = \left\{ \begin{array}{cc} 
				0 & 1 \\ 
				0.6287 & 0.3713 \\
				\end{array} \right\}$ 			
\end{tabular}
\end{tabular}}
\caption{Optimal designs for Cases 1, 2 and 3}
\label{tab:designs_all}
\end{table}

Table \ref{tab:designs_all} lists all $\Ds$- and KL-optimal designs we consider in our experiments.
In the first two Cases, the  ${\rm D}_1$-optimal design is supported at 0 and 1 as a consequence of Corollary \ref{cor:1}. In Case 1, all the KL-optimal designs share the same support points, that are 0 and 1. This behavior can be explained by Theorem \ref{th:KL}: the monotonicity of the variance function implies that $\underline{x}$ and $\overline{x}$ coincide with the extremes of the experimental domain $[0;1]$.
In Case 2, instead, the variance function is not monotone, for $\gamma>0.2755$, and thus
it does not necessarily reach its maximum and minimum at the extremes of $[0;1]$. When the support of $\xi^{KL}_{\gamma}$ is not constant with respect to $\gamma$, the limit result in Theorem \ref{th:limitKL} does not necessarily apply, although it is still true that $\omega(\gamma)>1/2$. 

Let us note that, from Table \ref{tab:designs_all} in both of the scalar cases as $\gamma={\lambda}/\sqrt{n}$ goes to $\gamma_0=0$, the KL-optimal design approaches the ${\rm D}_1$-optimal design, \red{which is the design which maximizes the noncentrality parameter (see Remark 1)}; this is consistent with \red{Theorem \ref{th:1} (Remark 1) and Corollary \ref{Corollary:2}} that imply asymptotic equivalence of ${\rm D}_1$- and KL-criteria. This is not true for Case 3, which concerns a multi-dimensional parameter $\bm{\gamma}$; in this case, the $\Ds$-optimal design has three support points while the KL-optimum design is always supported at 0 and 1.  
Note that in our settings, for Case 3, $\xi_{(\gamma_1,\gamma_2)}^{KL}=\xi_{(\gamma_2,\gamma_1)}^{KL}$. For this reason we do not consider redundant couples in our outcomes.

In Tables \ref{tab:h1_pows_05}-\ref{tab:finite_sample_case3} we display the experimental results for Cases 1, 2 and 3, respectively. We present Monte Carlo estimates of the actual significance level $\hat{\alpha}$ and the finite test powers $\hat{\eta}$ of the log-likelihood ratio (LR) test, for different optimality criteria in several experimental settings.
With a slight abuse of notation, the rows corresponding to $n=\infty$ display the nominal significance level $\alpha$ and the asymptotic test powers.
Let us note that, the asymptotic powers have been analytically computed from:
$$
{\rm P}\left( X^2_s[\zeta(\xi;\bm{\lambda};\bm{\gamma}_0)]\geq {\cal X}^2_{s;1-\alpha}  \right),
$$
where  ${\cal X}^2_{s;1-\alpha} $ is the quantile of a chi-squared random variable (rv) with $s$ degrees of freedom (df), $\alpha=0.05$, and $X_s^2[\zeta(\xi;\bm{\lambda};\bm{\gamma}_0)]$ is a non-central  chi-squared rv with $s$ df and noncentrality parameter $\zeta(\xi;\bm{\lambda};\bm{\gamma}_0)$ given by (\ref{noncentrality_parameter}).
Finite sample powers are computed as the proportions in the $M$ replications of LR-statistics larger than the asymptotic threshold ${\cal X}^2_{s;1-\alpha}  $.
In this paper only tables for $\alpha=0.05$ are included. Different significant levels produce similar results.

The values of the asymptotic powers confirm the theoretical results described in Remark 1 and \red{ Corollary \ref{Corollary:2}}. From Table \ref{tab:h1_pows_05} and \ref{tab:h2_pows_05} we can observe that the  ${\rm D}_1$-optimal design, which maximizes the noncentrality parameter for any value of $\lambda$ in the scalar case, actually guarantees the largest asymptotic powers in all the scenarios.
 \red{  On the other hand, when $\bm{\lambda}$ is a vector, 
from Table \ref{tab:finite_sample_case3} (see the rows  corresponding to $\bm{\gamma}^T=(0,0) $), we have that the design with the largest asymptotic power (for any choice of $\bm{\lambda}$), is the limiting KL-optimal design, $\xi_{\bm{\gamma}_0}^{KL}=\xi_{\bm{0}}^{KL}$ and  this is consistent with  Corollary \ref{Corollary:2}. }

To analyze the finite sample performances of tests, we generated, for each design and for each scenario (choice of $h$, $\vett \lambda$ and $n$), $M=10000$ Monte Carlo replications of samples of increasing sizes. Specifically, we chose $n=25,100,400$ for Case 1 and Case 3, whereas $n=100,400,1600,6400,25600$ for Case 2. In fact, the non-monotonicity of the conditional variance function of Case 2 makes it quite difficult to discriminate between the two hypotheses in (\ref{system}), thus a larger sample size is needed to obtain some non-zero power.
Several alternatives are considered, defined by $\lambda=\{5,10,20\}$ in Case 1, $\lambda=\{10,20,40\}$ in Case 2 and $\lambda_i=\{2.5, 5, 10\}$, with $i=1,2$, in Case 3 (removing redundant couples).

\subparagraph*{Case 1} Besides $\xi_{{\rm D}_1}$ and $\xi_{U_5} $, we considered the KL-optimal designs, $\xi_\gamma^{KL}$, for $\gamma=\lambda/\sqrt{n}$  in $\{0.25,\,0.5,\, 1,\, 2\, 4\}$.

\begin{table}[h]
\centering
\resizebox{.99\textwidth}{!}{%
\begin{tabular}{ccc|cc|cc|cc}
\multirow{2}{*}{$\lambda$} & \multirow{2}{*}{$n$} & \multirow{2}{*}{$\gamma$} & \multicolumn{2}{c}{$\xi^{KL}_{\gamma}$} & \multicolumn{2}{c}{$\xi_{{\rm D}_1}$ } & \multicolumn{2}{c}{$\xi_{U_5} $} \\
 & & & $\hat{\alpha}$  & \multicolumn{1}{c|}{$\hat{\eta}$} & $\hat{\alpha}$  & \multicolumn{1}{c|}{$\hat{\eta}$} & $\hat{\alpha}$  & \multicolumn{1}{c}{$\hat{\eta}$} \\
 \hline
 \multirow{4}{*}{5}  & 25       & 1    & 0.0680 & 0.3777 & 0.0672 & 0.4090 & 0.0686 & 0.2362  \\
  				     & 100      & 0.5  & 0.0555 & 0.4185 & 0.0509 & 0.4280 & 0.0548 & 0.2342  \\
  				     & 400      & 0.25 & 0.0463 & 0.4334 & 0.0500 & 0.4203 & 0.0488 & 0.2434  \\
   				     & $\infty$ & 0    & 0.0500 & 0.7054 & 0.0500 & 0.7054 & 0.0500 & 0.4239  \\

\hline
 \multirow{4}{*}{10} & 25       & 2    & 0.0674 & 0.8813 & 0.0672 & 0.9196 & 0.0672 & 0.6464  \\
  				     & 100      & 1    & 0.0517 & 0.9304 & 0.0509 & 0.9371 & 0.0548 & 0.6937  \\
  				     & 400      & 0.5  & 0.0517 & 0.9353 & 0.0500 & 0.9446 & 0.0488 & 0.7035  \\
   				     & $\infty$ & 0    & 0.0500 & 0.9988 & 0.0500 & 0.9988 & 0.0500 & 0.9424  \\
  				     
\hline
 \multirow{4}{*}{20} & 25       & 4    & 0.0863 & 0.9962 & 0.0672 & 1.0000 & 0.0672 & 0.9886  \\
  				     & 100      & 2    & 0.0591 & 1.0000 & 0.0509 & 1.0000 & 0.0548 & 0.9976  \\
  				     & 400      & 1    & 0.0557 & 1.0000 & 0.0500 & 1.0000 & 0.0488 & 0.9990  \\
   				     & $\infty$ & 0    & 0.0500 & 1.0000 & 0.0500 & 1.0000 & 0.0500 & 1.0000  \\

\hline
\end{tabular}
}
\caption{\small Monte Carlo actual significance level $\hat{\alpha}$ and finite powers $\hat{\eta}$, and their asymptotic counterparts,  obtained using different optimality criteria for Case 1, with variance function $h(x;\gamma)=e^{\gamma x}$, in various experimental settings. } 
\label{tab:h1_pows_05}
\end{table}

From Table \ref{tab:h1_pows_05}, which summarizes the finite sample and asymptotic results under Case 1, we may appreciate that, coherently with the asymptotic findings, the ${\rm D}_1$-optimal design  provides the highest finite power almost uniformly, with just one exception  where $\xi^{KL}_{\gamma}$ is slightly better. In addition, we can notice how, consistently \red{ with Remark 1, Corollary \ref{Corollary:2} and Theorem \ref{th:limitKL} }, the more $\gamma$ decreases the smaller the difference between KL- and ${\rm D}_1$-optimal designs.

In this example,  $h(x;\gamma)=e^{\gamma x}$ is a \emph{well-behaved} variance function; powers increase rapidly both with $n$ and $\lambda$, although for $\lambda=5$ it seems that a sample size of $n=400$ is not large enough to approach the asymptotic power. 
Our benchmark, the uniform design, appears to be the worst choice in all scenarios.
For all criteria, the actual significance level converges to the nominal level, displayed on rows $n=\infty$, as $n$ increases and it is relatively close to it already for $n=25$. In particular, design $\xi_{{\rm D}_1}$, shows the fastest convergence.

\subparagraph*{Case 2}

As  previously pointed out, when the  variance function is $h(x;\gamma)=1+0.1[\gamma x +\sin(2\pi\gamma x)]$, powers of the LR-statistic are very low, and tend to increase quite slowly with $n$ and $\lambda$, as displayed in Table \ref{tab:h2_pows_05}.

\begin{table}[h]
\centering
\resizebox{.99\textwidth}{!}{%
\begin{tabular}{ccc|cc|cc|cc}
\multirow{2}{*}{$\lambda$} & \multirow{2}{*}{$n$} & \multirow{2}{*}{$\gamma$} & \multicolumn{2}{c}{$\xi^{KL}_{\gamma}$} & \multicolumn{2}{c}{$\xi_{{\rm D}_1}$ } & \multicolumn{2}{c}{$\xi_{U_5} $} \\
 & & & $\hat{\alpha}$  & $\hat{\eta}$ & $\hat{\alpha}$  & $\hat{\eta}$ & $\hat{\alpha}$  & $\hat{\eta}$ \\\hline
 \multirow{6}{*}{10} & 100      & 1       & 0.0000 & 0.0003 & 0.0000 & 0.0021 & 0.0000 & 0.0005  \\
  				     & 400      & 0.5     & 0.0081 & 0.0305 & 0.0085 & 0.0108 & 0.0036 & 0.0120  \\
  				     & 1600     & 0.25    & 0.0480 & 0.3743 & 0.0482 & 0.3709 & 0.0596 & 0.2222  \\
  				     & 6400     & 0.125   & 0.0506 & 0.6096 & 0.0492 & 0.6159 & 0.0572 & 0.3921  \\
  				     & 25600    & 0.0625  & 0.0489 & 0.6916 & 0.0514 & 0.6943 & 0.0512 & 0.4049  \\
   				     & $\infty$ & 0       & 0.0500 & 0.9537 & 0.0500 & 0.9537 & 0.0500 & 0.7307  \\

\hline
 \multirow{6}{*}{20} & 100   & 2      & 0.0000 & 0.0217 & 0.0000 & 0.0389 & 0.0000 & 0.0090  \\
  				     & 400   & 1      & 0.0000 & 0.0899 & 0.0085 & 0.0814 & 0.0036 & 0.0739  \\
  				     & 1600  & 0.5    & 0.0493 & 0.3819 & 0.0482 & 0.0963 & 0.0596 & 0.2709  \\
  				     & 6400  & 0.25   & 0.0486 & 0.9133 & 0.0501 & 0.9125 & 0.0596 & 0.7274  \\
  				     & 25600 & 0.125  & 0.0518 & 0.9939 & 0.0514 & 0.9937 & 0.0541 & 0.8952  \\
   				     & $\infty$ & 0   & 0.0500 & 1.0000 & 0.0500 & 1.0000 & 0.0500 & 0.9993  \\
  				     
\hline
 \multirow{6}{*}{40} & 100    & 4     & 0.0000 & 0.2236 & 0.0000 & 0.1945 & 0.0000 & 0.0635  \\
  				     & 400    & 2     & 0.0000 & 0.3543 & 0.0085 & 0.2547 & 0.0036 & 0.1579  \\
  				     & 1600   & 1     & 0.0109 & 0.5399 & 0.0482 & 0.2660 & 0.0596 & 0.4365  \\
  				     & 6400   & 0.5   & 0.0521 & 0.9206 & 0.0501 & 0.2932 & 0.0596 & 0.7294  \\
   				     & 25600  & 0.25  & 0.0492 & 1.0000 & 0.0514 & 1.0000 & 0.0541 & 0.9955  \\
   				     & $\infty$ & 0   & 0.0500 & 1.0000 & 0.0500 & 1.0000 & 0.0500 & 1.0000  \\

\hline
\end{tabular}
}
\caption{\small Monte Carlo actual significance level $\hat{\alpha}$ and finite powers $\hat{\eta}$, and their asymptotic counterparts,  obtained using different optimality criteria for Case 2, with variance function $h(x;\gamma)=1+0.1[\gamma x +\sin(2\pi\gamma x)]$, in various experimental settings.} 
\label{tab:h2_pows_05}
\end{table}

It is  worth noting that the actual significance level gets close to the nominal level only for $n\geq 1600$, for all designs.
This means that the asymptotic critical values used to determine the rejection region are too high, which reflects in a poor performance of the LR-test for small samples in all scenarios.
In general, it seems that KL-optimal designs provide higher finite power faster (i.e. for lower $n$) with respect to the ${\rm D}_1$-optimal design.
Further, in our experiments, only for $n=25600$ and $\gamma=0.0625$ the ${\rm D}_1$-optimal design shows its asymptotic dominance over its competitors, while in the other cases it is dominated by the KL-optimal design.
Unfortunately, the finite sample power of KL-designs is strongly dependent on the knowledge of the alternative and  thus it is difficult to suggest a rule of thumb for the choice of $\gamma$ in case of uncertainty. To this extent we suggest the reading of \cite{Fidalgo_etal2010}.

\subparagraph*{Case 3} 
Table \ref{tab:finite_sample_case3} displays the finite powers and actual significance level (and their asymptotic counterparts) of the LR-tests associated to the different designs for variance function $h(x;\bm{\gamma})=1+\gamma_1 x+ \gamma_2 x^2$.
We recall that, for this particular specification of $h(\cdot;\cdot)$,  $\xi_{(\gamma_1,\gamma_2)}^{KL}=\xi_{(\gamma_2,\gamma_1)}^{KL}$, and thus we consider only the distinct KL-optimal designs.

\begin{table}[ht]
\centering
\resizebox{.99\textwidth}{!}{%

\begin{tabular}{ccc|cc|cc|cc}
\multirow{2}{*}{$\bm{\lambda}^T$} & \multirow{2}{*}{$n$} & \multirow{2}{*}{$\bm{\gamma}^T$} & \multicolumn{2}{c}{$\xi^{KL}_{\bm{\gamma}}$} & \multicolumn{2}{c}{$\xi_{{\rm D}_2}$ } & \multicolumn{2}{c}{$\xi_{U_5} $} \\
 & & & $\hat{\alpha}$  & $\hat{\eta}$ & $\hat{\alpha}$  & $\hat{\eta}$ & $\hat{\alpha}$  & $\hat{\eta}$ \\
 \hline														      
 \multirow{3}{*}{(2.5, 2.5)} & 25         & (0.5   , 0.5    ) & 0.0202 & 0.1104 & 0.0782 & 0.1522 & 0.0973 & 0.1529  \\
							 & 100        & (0.25  , 0.25   ) & 0.0168 & 0.1532 & 0.0556 & 0.1679 & 0.0634 & 0.1457  \\
  				     		 & 400        & (0.125 , 0.125  ) & 0.0129 & 0.1861 & 0.0535 & 0.1980 & 0.0532 & 0.1661  \\
                             & $\infty$   & (0     , 0      ) & 0.0500 & 0.6028 & 0.0500 & 0.4384 & 0.0500 & 0.3400  \\
  				     \hline
 \multirow{3}{*}{(2.5,   5)} & 25         & (0.5   , 1      ) & 0.0218 & 0.1837 & 0.0782 & 0.2160 & 0.0973 & 0.2095  \\
							 & 100        & (0.25  , 0.5    ) & 0.0194 & 0.3083 & 0.0556 & 0.2802 & 0.0634 & 0.2445  \\
  				     		 & 400        & (0.125 , 0.25   ) & 0.0146 & 0.4165 & 0.0535 & 0.3598 & 0.0532 & 0.2906  \\
                             & $\infty$   & (0     , 0      ) & 0.0500 & 0.9291 & 0.0500 & 0.8038 & 0.0500 & 0.6753  \\
	   				     \hline
 \multirow{3}{*}{(5,  5)   } & 25         & (1     , 1     ) & 0.0206 & 0.2697 & 0.0782 & 0.2766 & 0.0973 & 0.2459  \\
							 & 100        & (0.5   , 0.5   ) & 0.0157 & 0.4855 & 0.0556 & 0.4211 & 0.0634 & 0.3264  \\
  				     		 & 400        & (0.25  , 0.25  ) & 0.0133 & 0.6579 & 0.0535 & 0.5444 & 0.0532 & 0.4372  \\
                             & $\infty$   & (0     , 0     ) & 0.0500 & 0.9965 & 0.0500 & 0.9668 & 0.0500 & 0.9028  \\
\hline
 \multirow{4}{*}{(2.5, 10)}  & 25         & (0.5   , 2     ) & 0.0240 & 0.3744 & 0.0782 & 0.3418 & 0.0973 & 0.3143  \\
							 & 100        & (0.25  , 1     ) & 0.0158 & 0.6458 & 0.0556 & 0.5596 & 0.0634 & 0.4500  \\
  				     		 & 400        & (0.125 , 0.5   ) & 0.0159 & 0.8361 & 0.0535 & 0.7241 & 0.0532 & 0.5969  \\
                             & $\infty$   & (0     , 0     ) & 0.0500 & 1.0000 & 0.0500 & 0.9983 & 0.0500 & 0.9876  \\
  				     		   				     \hline
 \multirow{4}{*}{(5, 10)}    & 25         & (1     , 2     ) & 0.0221 & 0.4560 & 0.0782 & 0.4119 & 0.0973 & 0.3558  \\
							 & 100        & (0.5   , 1     ) & 0.0163 & 0.7638 & 0.0556 & 0.6509 & 0.0634 & 0.5378  \\
  				     		 & 400        & (0.25  , 0.5   ) & 0.0149 & 0.9280 & 0.0535 & 0.8376 & 0.0532 & 0.7131  \\
  				     		 & $\infty$   & (0     , 0     ) & 0.0500 & 1.0000 & 0.05   & 1.0000 & 0.05   & 0.9990  \\  				     		   				     \hline
 \multirow{4}{*}{(10, 10)}   & 25         & (2     , 2     ) & 0.0231 & 0.5664 & 0.0782 & 0.5232 & 0.0973 & 0.4285  \\
							 & 100        & (1     , 1     ) & 0.0163 & 0.9067 & 0.0556 & 0.8084 & 0.0634 & 0.6919  \\
  				     		 & 400        & (0.5   , 0.5   ) & 0.0135 & 0.9908 & 0.0535 & 0.9558 & 0.0532 & 0.8835  \\
                             & $\infty$   & (0     , 0     ) & 0.0500 & 1.0000 & 0.0500 & 1.0000 & 0.0500 & 1.0000  \\
  				     		  				     \hline		       				       				      				     
\end{tabular}
			
}
\caption{\small Monte Carlo actual significance level $\hat{\alpha}$ and finite powers $\hat{\eta}$, and their asymptotic counterparts,  obtained using different optimality criteria for Case 2, with variance function $h(x;\bm{\gamma})=1+\gamma_1 x+ \gamma_2 x^2$, in various experimental settings.
\label{tab:finite_sample_case3}
}
\end{table}

We can immediately notice that the KL-optimal designs present very low actual significance level, implying a diffuse under-rejection rate; differently, both  $\xi_{{\rm D}_2}$ and $\xi_{U_5}$ have a null rejection rate near to the nominal level, $\alpha=0.05$, and tend to over-reject for small samples.
The lower actual significance level implies that a downward correction should be done to the critical values, that would positively affect the powers of the test statistics under the KL-designs. Yet notwithstanding this, except for the  case $\bm{\lambda}=(2.5,2.5)$ the highest powers of the LR-test are those associated to the KL-optimal designs. This is in line with the asymptotic theory, suggesting that for small values of $\|\vett\gamma\|$ the KL-criterion tends to coincide with the noncentrality parameter of the asymptotic chi-squared distribution under the alternative (Theorem \ref{th:3}).

\section{Conclusions}

The main goal of this paper is the study of optimality criteria for detecting heteroscedasticity in a non-linear Gaussian regression model, under local alternatives. Three  tests of hypotheses are typically used for discriminating between nested rival models: the Likelihood Ratio, the Score and the Wald tests. All of them are related to the likelihood function (thus, they are referred to likelihood-based tests) and, under local alternatives, they share asymptotically the same noncentral chi-squared distribution.
On the other hand, from the experimental design perspective,  D$_s$- and KL-optimalities are the most common criteria to design experiments for discrimination. 

In this paper, we prove theoretically to what extent both these criteria are related to the noncentrality parameter of a likelihood-based test for discriminating heteroscedasticiy versus homoscedasticity. 

The KL-criterion is proved to be asymptotically equivalent to the noncentrality parameter. Therefore, asymptotically the KL-optimum design guarantees the maximum power of the likelihood-based tests. The D$_s$-criterion instead is proportional to the noncentrality parameter  whenever the variance function depends just on one parameter (even for finite samples, not only asymptotically). Therefore, in this last case, the two criteria of optimality are asymptotically equivalent. 

The numerical outcomes obtained from a simulation study confirm these theoretical results.

\appendix
\section*{Appendix: proofs and technical results}
\label{Appendix}
{\bf Theorem \ref{th:Ds_to_D_aux}.}
\begin{proof}
	\noindent 
	The $\Ds$-optimum design for $\bm{\gamma}$ maximizes the following criterion (see for instance, \cite{Atkinson_etal2007}, Sect.10.3):
	\begin{equation}
	\Phi_{\Ds}{(\xi; \bm{\beta},\sigma^2, \bm{\gamma})}
	=
	\frac{|\mathcal{I}(\xi;\bm{\beta},\sigma^2, \bm{\gamma})|}{|\mathcal{I}_{11}(\xi;\bm{\beta},\sigma^2, \bm{\gamma})|}=
	\frac{|\mathcal{I}(\xi;\bm{\beta},\sigma^2, \bm{\gamma})|}{\displaystyle\frac{1}{2 \sigma^4}\,|\bm{M}(\xi;\bm{\beta},\sigma^2, \bm{\gamma})|}.
	\label{Ds}
	\end{equation}
	Since the Fisher information matrix can be partitioned as follows:
	\begin{equation}
	\mathcal{I}(\xi;\bm{\beta},\sigma^2, \bm{\gamma}) =
	\begin{bmatrix}
	\bm{M}(\xi;\bm{\beta},\sigma^2, \bm{\gamma}) & \bm{0}\\
	\bm{0} & \frac{1}{2\sigma^4} {\bm{V}}(\xi;\sigma^2, \bm{\gamma})
	\end{bmatrix},
	\label{I}
	\end{equation}
	where
	$$
	{\bm{V}}(\xi;\sigma^2\!, \bm{\gamma})\!=\!\! \sum_{i=1}^k \!\bm{f}(\bm{x}_i;\sigma^2\!,\!\bm{\gamma}) \bm{f}(\bm{x}_i;\sigma^2\!,\!\bm{\gamma})^T
	\xi(\bm{x}_i), \quad 
	\bm{f}(\bm{x};\!\sigma^2\!,\!\bm{\gamma})\!=\!\!
	\begin{bmatrix}
	1 \\
	\sigma^2 \,\nabla\!\log h(\bm{x};\bm{\gamma})
	\end{bmatrix},
	$$
	the $\Ds$-criterion (\ref{Ds}) becomes,
	\begin{equation}
	\Phi_{\Ds}{(\xi;\sigma^2, \bm{\gamma})}
	=
	\left(\frac{1}{2 \sigma^4}\right)^s|{\bm{V}}(\xi;\sigma^2, \bm{\gamma})|
	=
	\left(\frac{1}{2 }\right)^s|{\widetilde{\bm{V}}}(\xi; \bm{\gamma})|
	\label{Ds_1}
	\end{equation}
	where 
	$$
	\widetilde{\bm{V}}(\xi; \bm{\gamma})= \sum_{i=1}^k \tilde{\bm{f}}(\bm{x}_i;\bm{\gamma}) \tilde{\bm{f}}(\bm{x}_i;\bm{\gamma})^T
	\, \xi(\bm{x}_i), \quad 
	\tilde{\bm{f}}(\bm{x};\bm{\gamma})\!=\!
	\begin{bmatrix}
	1 \\
	\nabla\!\log h(\bm{x};\bm{\gamma})
	\end{bmatrix},
	$$
	which proves the theorem.
\end{proof}
\noindent{\bf Corollary \ref{cor:1}.}
\begin{proof}
The proof follows immediately from Theorem \ref{th:Ds_to_D_aux}.
	Since $\gamma$ is scalar, 
	\begin{align*}
	|\tilde{\bm{V}}|
	& =  \displaystyle\sum_{i=1}^k \left[
	\frac{\nabla h(x_i;\gamma)}{h(x_i;\gamma)}
	\right]^2 \xi(\bm{x}_i) - 
	\left[
	\displaystyle\sum_{i=1}^k{\frac{ \nabla h(x_i;\gamma)}{h(x_i;\gamma)}\, \xi(\bm{x}_i)}
	\right]^2 
	= 
	\operatorname{Var} \left( \nabla \log h(x_i;\gamma)\right),
	\end{align*}
	where $\operatorname{Var(\cdot)}$ denotes the variance with respect to $\xi$. Therefore, a ${\rm D}_1$-optimal design for $\gamma$  is equally weighted at ${x}_l\in \{x\!: \operatorname*{argmin}_x \nabla \log h(x;\gamma)\}$ and ${x}_u\in \{x\!: \operatorname*{argmax}_x \nabla \log h(x;\gamma)\}$.
\end{proof}

\noindent{\bf Theorem \ref{thm:ex_lem1}.}
\begin{proof}
	Differentiating the right-hand side of (\ref{I12}) wrt $\bm{\beta}$ and setting it equal to zero we obtain $\hat{\bm{\beta}}=\bm{\beta}_1$ and the minimum Kullback-Leibler distance becomes:
	\begin{equation}
	\min_{\sigma^2} \sum_{i=1}^k  \left[\frac{\sigma_1^2 h(\bm{x}_i)}{\sigma^2} -\log \frac{\sigma_1^2 h(\bm{x}_i)}{\sigma^2} \right] \xi(\bm{x}_i).
	\label{min_KL-dist}
	\end{equation}
	Differentiating the function to be minimized in (\ref{min_KL-dist}) wrt $\sigma^2$ and setting it equal to zero: 
	$$
	\frac{\partial }{\partial \sigma^2}  \sum_{i=1}^k \left[
	\frac{\sigma_1^2 h(\bm{x}_i)}{\sigma^2}-\log [\sigma_1^2 h(\bm{x}_i)]+\log \sigma^2
	\right]\xi(\bm{x}_i) =0,
	$$
	we obtain
	$$
	\hat\sigma^2=\sigma_1^2 \sum_{i=1}^k h(\bm{x}_i) \;\xi(\bm{x}_i)
	$$
	and thus the KL-criterion (\ref{min_KL-dist}) becomes:
	\begin{eqnarray}
	I_{12}(\xi;\bm{\gamma}_1) &=&
	1+\log \sum_{i=1}^k  h(\bm{x}_i) \,\xi(\bm{x}_i)-\sum_{i=1}^k \log h(\bm{x}_i) \,\xi(\bm{x}_i) \nonumber\\
	&=&
	1+\log \sum_{i=1}^k h(\bm{x}_i) \,\xi(\bm{x}_i)-\log \prod_{i=1}^{k}[h(\bm{x}_i)]^{\xi(\bm{x}_i)}.
	\label{KL}
	\end{eqnarray}
\end{proof}
\noindent{\bf Theorem \ref{th:3}.}
\begin{proof}
	Let $\bm{\gamma}_1={\bm{\gamma}}_0+\bm{\lambda}/\sqrt{n}$ for a specific value of $\bm{\lambda}$. If $n$ is large enough to guarantee $\bm{\gamma}_0+\bm{\lambda}/\sqrt{n} = \bm{\gamma}_1 \in \mathcal{B}(\bm{\gamma}_0,r)$, assumption (\ref{conditions_unif_continuity}) allows us to compute a first order Taylor expansion of the function $h(\bm{x};\bm{\gamma}_1)$ at ${\bm{\gamma}}_0$, and thus (\ref{eqn:I21}), becomes:
	\begin{eqnarray}\label{eq:Taylor_exp_h}
		I_{12}(\xi;\bm{\gamma}_1) &=&\!\!
		1+\log \left\{ \sum_{i=1}^k \Big[ h(\bm{x}_i;{\bm{\gamma}}_0)+\nabla h(\bm{x}_i;{\bm{\gamma}}_0)^T (\bm{\gamma}_1-{\bm{\gamma}}_0) \right. \nonumber\\
		&+&
		\left. \frac{1}{2}\,(\bm{\gamma}_1-{\bm{\gamma}}_0)^T h^{''}\!(\bm{x}_i;\bm{\bar\gamma})\,
		(\bm{\gamma}_1-{\bm{\gamma}}_0) \Big] \xi(\bm{x}_i)  \right\}
		\nonumber\\
		&-&
		\sum_{i=1}^k \xi(\bm{x}_i) \log \Big[ h(\bm{x}_i;{\bm{\gamma}}_0) + \nabla h(\bm{x}_i;{\bm{\gamma}}_0)^T (\bm{\gamma}_1-{\bm{\gamma}}_0)
		\nonumber\\
		& +&
		\frac{1}{2}\,(\bm{\gamma}_1-{\bm{\gamma}}_0)^T h^{''}\!(\bm{x}_i;\bm{\bar\gamma})\,
		(\bm{\gamma}_1-{\bm{\gamma}}_0) \Big]   \nonumber\\
		&=&
		1\!+\!\log \left\{\! 1+\sum_{i=1}^k \!\left[\nabla h(\bm{x}_i;{\bm{\gamma}}_0)^T \frac{\bm{\lambda}}{\sqrt{n}} + \frac{1}{2\,n}\bm{\lambda}^T h^{''}\!(\bm{x}_i;\bm{\bar\gamma})
		\bm{\lambda} \right] \!\xi(\bm{x}_i) \! \right\}
		\nonumber\\
		&-&
		\sum_{i=1}^k \xi(\bm{x}_i) \log \left[1+\nabla h(\bm{x}_i;{\bm{\gamma}}_0)^T \frac{\bm{\lambda}}{\sqrt{n}} + \frac{1}{2\,n}\bm{\lambda}^T h^{''}\!(\bm{x}_i;\bm{\bar\gamma})
		\bm{\lambda} \right],   
	\end{eqnarray}
	where $h^{''}\!(\bm{x}_i;\bm{\gamma})$ denotes the Hessian matrix of $h(\bm{x}_i;\bm{\gamma})$ and $\bm{\bar\gamma}$ is a value of $\bm{\gamma}$ such that $||{\bm{\gamma}}_0-{\bar{\gamma}}||\leq ||{\bm{\gamma}}_0-{\bm{\gamma}}_1||$.
	From the Taylor expansion of $\log(1+z)$ up to the second order, the previous expression becomes:
	\begin{eqnarray*}
		I_{12}(\xi;\bm{\gamma}_1) &=&
		1\!+\! \sum_{i=1}^k \!\left[\nabla h(\bm{x}_i;{\bm{\gamma}}_0)^T \frac{\bm{\lambda}}{\sqrt{n}} + \frac{1}{2\,n}\bm{\lambda}^T h^{''}\!(\bm{x}_i;\bm{\bar\gamma})
		\bm{\lambda} \right] \!\xi(\bm{x}_i)\\
		&-&\!\!\!
		\frac{1}{2}
		\left\{\! \sum_{i=1}^k \!\left[\nabla h(\bm{x}_i;{\bm{\gamma}}_0)^T \frac{\bm{\lambda}}{\sqrt{n}}\! +\! \frac{1}{2\,n}\bm{\lambda}^T h^{''}\!(\bm{x}_i;\bm{\bar\gamma})
		\bm{\lambda} \right] \!\xi(\bm{x}_i) \! \right\}^2 \!\!\!+\! O\!\left(\dfrac{||\bm{\lambda}||^3}{n^{\frac{3}{2}}}\right)
		\\
		&-&
		\sum_{i=1}^k \xi(\bm{x}_i) \left\{
		\nabla h(\bm{x}_i;{\bm{\gamma}}_0)^T \frac{\bm{\lambda}}{\sqrt{n}} + \frac{1}{2\,n}\bm{\lambda}^T h^{''}\!(\bm{x}_i;\bm{\bar\gamma}) \bm{\lambda} \right. \\
		&-&
		\left.
		\frac{1}{2}
		\left[ \nabla h(\bm{x}_i;{\bm{\gamma}}_0)^T \frac{\bm{\lambda}}{\sqrt{n}} + \frac{1}{2\,n}\bm{\lambda}^T h^{''}\!(\bm{x}_i;\bm{\bar\gamma}) \bm{\lambda} \right]^2
		+O\!\left(\dfrac{||\bm{\lambda}||^3}{n^{\frac{3}{2}}}\right)
		\right\}   
	\end{eqnarray*}	
	The first terms of the two developments simplify and thus:
	\begin{eqnarray*}
		I_{12}(\xi;\bm{\gamma}_1) &=&
		1-
		\frac{1}{2}
		\left[\sum_{i=1}^k \!\nabla h(\bm{x}_i;{\bm{\gamma}}_0)^T \frac{\bm{\lambda}}{\sqrt{n}} \,\xi(\bm{x}_i) \right] ^2 
		\\
		&+&
		\frac{1}{2}
		\sum_{i=1}^k \xi(\bm{x}_i) 
		\left[ \nabla h(\bm{x}_i;{\bm{\gamma}}_0)^T \frac{\bm{\lambda}}{\sqrt{n}}  \right]^2 
		+O\!\left(\dfrac{||\bm{\lambda}||^3}{n^{\frac{3}{2}}}\right)
		\\  
		&=&
		1-
		\frac{1}{2\,n}
		\left[\bm{\lambda}^T \sum_{i=1}^k \!\nabla h(\bm{x}_i;{\bm{\gamma}}_0) \,\xi(\bm{x}_i)  \right]
		\left[\sum_{i=1}^k \!\nabla h(\bm{x}_i;{\bm{\gamma}}_0)^T \,\xi(\bm{x}_i) \bm{\lambda} \right]  
		\\
		&+&
		\frac{1}{2\,n}
		\sum_{i=1}^k \xi(\bm{x}_i) 
		\left[\bm{\lambda}^T \nabla h(\bm{x}_i;{\bm{\gamma}}_0)  \right] \left[ \nabla h(\bm{x}_i;{\bm{\gamma}}_0)^T \bm{\lambda}  \right]
		+ O\!\left(\dfrac{||\bm{\lambda}||^3}{n^{\frac{3}{2}}}\right)
	\end{eqnarray*}
	Rearranging the terms and taking into account Expression (\ref{noncentrality_parameter}) for the noncentrality parameter, we obtain the thesis.
\end{proof}  

Next Lemma is used for the proof of Theorem \ref{th:KL}
\begin{lem}\label{lem:bounds}
	Let $h_i$, $i=1,\ldots, n$ be $n$ positive quantities such that $$0<\underline{h}= h_1 \leq h_2 \leq \cdots \leq h_n = \overline{h}<\infty,$$ and let $A_n= n^{-1}\sum_i h_i $ and $G_n= \sqrt[n]{\prod_i h_i }$ the arithmetic and geometric means of $\{h_i\}_{i\leq n}$.
	We have the following bounds for $D_n=\log A_n-\log G_n$:
	\begin{equation}
	\label{eqn:Dn}
	\frac{2}{n}\left[\log \frac{\underline{h}+\overline{h}}{2}-\frac{1}{2} (\log \underline{h}+\log \overline{h})\right]\leq D_n \leq -\log H + H - 1,
	\end{equation}
	where
	\begin{equation}
	\label{H}
	H=\underline{h}\cdot  \frac{\log \overline{h}-\log \underline{h}}{\overline{h}-\underline{h}}.
	\end{equation}
	The lower bound is achieved when $h_1=\underline{h}$, $h_{n}=\overline{h}$ and $h_i=(\overline{h}+\underline{h})/2$ for $i=2,\ldots,n-1$.
	Further, 
	the upper bound is reached for $h_1=\cdots=h_{r^*}=\underline{h}$ and $h_{r^*+1}=\cdots = h_n=\overline{h}$, where $r^*=\left\lfloor n\left(\frac{\overline{h}}{\overline{h}-\underline{h}}-\frac{1}{\log \overline{h}-\log\underline{h}}\right)\right\rfloor $ and $\lfloor \cdot \rfloor$ denotes the integer part of a number.
\end{lem}

\begin{proof}
	\begin{description}
		\item[a) Lower Bound] 
		To find the lower bound in (\ref{eqn:Dn}), we use a recursive argument similar to \cite{Tung1975}.
		Given the geometric and arithmetic means of  $n-1$ terms, $A_{n-1}$ and $G_{n-1}$, we can define the difference $D_n$ as a function of an \lq\lq additional'' term $a$:
		\begin{equation}
		\label{Dn}
		D_n(a)=\log\left[n^{-1}((n-1)A_{n-1}+a)\right]- n^{-1}\log \left(G_{n-1}^{n-1} \cdot a\right)
		\end{equation}
		then we solve
		\begin{equation}
		\label{D'}
		D_n'(a)=\frac{1}{(n-1)A_{n-1}+a}-\frac{1}{n\,a}=0
		\end{equation}
		that gives $a=A_{n-1}$. Since $D''(A_{n-1})=\frac{n-1}{n\, A_{n-1}}>0$ then $D_n\geq D_n(A_{n-1})=\frac{n-1}{n} D_{n-1}$ is the minimum.
		By applying recursively, we get:
		$$D_n \geq \frac{n-1}{n} D_{n-1}
		\geq \cdots \geq \frac{2}{n} {\rm D}_2=\frac{2}{n}\left[ \log\left(\frac{\underline{h}+\overline{h}}{2}\right)-\frac{1}{2}\log(\underline{h}\cdot \overline{h})\right].$$
		It can be easily checked that $D_n=\frac{2}{n} {\rm D}_2$ if $h_1=\underline{h}$, $h_n=\overline{h}$ and all other terms are equal to $A_2=\frac{\underline{h}+\overline{h}}{2}$.
		
		\item[b) Upper Bound]  
		From left hand-side of (\ref{D'}), we have that $D'_n(a)\geq 0$ for all $a\geq A_{n-1}$ and thus, $D_n(a)$ reaches its maximum in one of the extremes $\underline{h}$ or $\overline{h}$.
		
		Let then $r$ be an integer such that 
		$$\underline{h}=h_1=\ldots=h_{r} < h_{r+1} = \ldots= \overline{h}.$$
		We can then define the function $D_n(r)$:
		
		\[\eqalign{D_n(r)&=\log \frac{r \underline{h}+(n-r)\overline{h}}{n} - \frac{r}{n}\log \underline{h}-\frac{n-r}{n}\log\overline{h}\cr
			&= \log \frac{r +(n-r) b}{n} -\frac{n-r}{n}\log b
		}
		\]
		where $b=\overline{h}/\underline{h}$.
		
		By extending the function $D_n(r)$ to the whole interval $[1,n]$, we can solve:
		$$D_n'(r)=\frac{1-b}{r+(n-r) b}+\frac{1}{n\log b}=0,$$
		that gives the maximum point $$r^*=n\left(\frac{b}{b-1}-\frac{1}{\log b}\right),$$
		with $$D_n(r^*)=\frac{\log b}{b-1}-1-\log\frac{\log b}{b-1}=H-1-\log H,$$
		where  $H$ is defined in (\ref{H}).
		
		If $r^*$ is not integer, then the maximum of $D_n(r)$ over $[1,n]\cap \mathbb{N}$ must be attained at either $\lfloor r^* \rfloor$ or $\lfloor r^* \rfloor + 1$.
		By writing $\lfloor r^* \rfloor=n\left(\frac{b}{b-1}-\frac{1}{\log b}\right)-\mu$, where $\mu \in (0,1)$ is the fractional part and 
		by comparing $D_n(\lfloor r^* \rfloor)$ and $D_n(\lfloor r^* \rfloor+1)$, we easily see that
		$$
		D_n(\lfloor r^* \rfloor)-D_n(\lfloor r^* \rfloor+1) =\frac{b-1}{n}-\log b \geq 0$$
		which concludes the proof.
	\end{description}
\end{proof}

\noindent{\bf Theorem \ref{th:KL}.}
\begin{proof}
	The maximization of (\ref{eqn:I21}) corresponds to finding the maximum possible value for the difference between the logarithms of the arithmetic and geometric means of $h(\vett x_i)$, $i=1,\ldots,k$.
	The proof follows immediately from part b) of the proof of Lemma \ref{lem:bounds}. 
\end{proof}

\noindent{\bf Theorem \ref{KL_asimpt}.}
\begin{proof}
	By setting $b=\overline{h}/\underline{h}$ in (\ref{KL-design}) we easily obtain 
	$$\omega=\frac{b\log b -b +1}{(b-1)\log b}.$$
	Taking the limit for $\bm{\gamma}_1\to \bm{\gamma}_0$, that is, for $b\to 1$, and by applying de l'H\"opital rule twice,
	$$\eqalign{\lim_{b\to 1}\frac{b\log b -b +1}{(b-1)\log b} &= \lim_{b\to 1} \frac{\log b}{\log b +(b-1)/b} =  \lim_{b\to 1} \frac{b\log b}{b\log b +b-1}\cr &=\lim_{b\to 1} \frac{\log b+1}{\log b +2} =\frac{1}{2}}.$$
	
	In order to complete the proof, it is enough to show that the function $\omega(\bm{\gamma}_1)-1/2>0$, or equivalently, that the function $$\frac{b}{b-1}-\frac{1}{\log (1+(b-1))}-1/2>0$$ for all $b>1$. 
	By denoting $v=b-1$, we can rewrite the above inequality as
	$$\log(1+v)>\frac{2v}{v+2}.$$
	Since we already proved that the function $g(v)=\log(1+v)-\frac{2v}{v+2}\to 0$ for $v\to 0$, it is enough to show that, for $v>0$, its derivative is positive:
	$$g'(v)=\frac{1}{v+1}-\frac{2v+4-2v}{(v+2)^2}=\frac{v^2}{(v+1)(v+2)^2}>0.$$
	
\end{proof}

\noindent{\bf Theorem \ref{th:unif_cont}.}
\begin{proof}
	Let $\bm{\gamma}_1={\bm{\gamma}}_0+\bm{\lambda}/\sqrt{n}$.
For convenience, we denote the mean with respect to a design $\xi$ by $\E_\xi$ and let	
\begin{equation}
Z_{in}=Z_n(\bm{x}_i,\bm{\lambda},\bm{\gamma}_0):=\nabla h(\bm{x}_i;{\bm{\gamma}}_0)^T \frac{\bm{\lambda}}{\sqrt{n}}+ \frac{1}{2n}\bm{\lambda}^T h^{''}\!(\bm{x}_i;\bm{\bar\gamma})
		\bm{\lambda}. 
		\label{Zn}
		\end{equation}

With this notation, we can rewrite equation (\ref{eq:Taylor_exp_h}) as
\begin{equation}
\label{eq:taylor_h}
I_{12}(\xi;\bm{\gamma}_1) = 1+ \log\left(1+ \E_\xi Z_n\right) - \E_\xi \log\left(1+Z_n\right). 
\end{equation}

We clearly have that,
\begin{eqnarray*}
\sup_{\bm{x} \in \mathcal{X}} |Z_n(\bm{x},\bm{\lambda},\bm{\gamma}_0)| & \leq &
\sup_{\bm{x} \in \mathcal{X}}  \left|\nabla h(\bm{x};{\bm{\gamma}}_0)^T \frac{\bm{\lambda}}{\sqrt{n}}\right| + \sup_{\bm{x} \in \mathcal{X}} \left|  \frac{1}{2n}\bm{\lambda}^T h^{''}\!(\bm{x};\bm{\bar\gamma})
		\bm{\lambda}\right|
\end{eqnarray*}

Now, from the Cauchy-Shwartz inequality and from the assumptions of the Theorem:
\begin{equation}
|\nabla h(\bm{x};{\bm{\gamma}}_0)^T \bm{\lambda}| \leq \|\nabla h(\bm{x};{\bm{\gamma}}_0)^T\| \cdot \|\bm{\lambda}\|\leq M_1 L
\label{ineq:1}
 \end{equation}
and, by the definition of the spectral norm:
\begin{eqnarray}
|\bm{\lambda}^T h^{''}\!(\bm{x};\bm{\bar\gamma})\bm{\lambda}| &\leq& \|\bm{\lambda}\| \|h^{''}\!(\bm{x};\bm{\bar\gamma})\bm{\lambda}\| \leq \|\bm{\lambda}\|^2 \sup_{\lambda \not=0} \frac{\|h^{''}\!(\bm{x};\bm{\bar\gamma})\bm{\lambda}\|}{\|\bm\lambda\|}
\nonumber
\\
&=& \|\bm{\lambda}\|^2  \|h^{''}\!(\bm{x};\bm{\bar\gamma})\|_2\leq M_2 L^2. 
\label{ineq:2}
\end{eqnarray}

Thus, we get the upper bound:
\begin{eqnarray*}
\sup_{\bm{x} \in \mathcal{X}} |Z_n(\bm{x},\bm{\lambda},\bm{\gamma}_0)| & \leq &
\frac{M_1   L}{\sqrt{n}}
 + \frac{M_2 L^2}{2n} 
\end{eqnarray*}
 where the constants $L,M_1,M_2$ do not depend on $\xi$ nor on $n$.
 
 It is therefore possible to find an integer $n^*:=n^*(L,M_1,M_2)$ such that, for all $n\geq n^*$: 
 \begin{equation}\label{nstar}
\max\left\{ \frac{M_1  L}{\sqrt{n}},\,\frac{ M_2 L^2}{2n}  \right\}< \frac{1}{4}
 \end{equation}

Thus, for $n\geq n^*$, we have that $\sup_{\bm{x} \in \mathcal{X}} |Z_n(\bm{x},\bm{\lambda},\bm{\gamma}_0)|<\frac{1}{2}$. As a consequence,  $Z_n<\frac{1}{2}$ and,  replacing each term of the expected value $\E_\xi |Z_n|$ with its maximum value, $$
\E_\xi Z_n \leq \E_\xi |Z_n| \leq \sup_{\bm{x} \in \mathcal{X}} |Z_n(\bm{x},\bm{\lambda},\bm{\gamma}_0)|<\frac{1}{2}. $$
This enables us to expand in infinite series both the terms $\log(1+\E_\xi Z_n)$ and $\log(1+Z_{n})$ in (\ref{eq:taylor_h}), obtaining:
	\begin{eqnarray*}
		I_{12}(\xi;\bm{\gamma}_1) 
		&=& 1+ \left[ \sum_{i=1}^k  \left(\nabla h(\bm{x}_i;{\bm{\gamma}}_0)^T \frac{\bm{\lambda}}{\sqrt{n}}+ \frac{1}{2n}\bm{\lambda}^T h^{''}\!(\bm{x}_i;\bm{\bar\gamma})
		\bm{\lambda} \right)\!\xi(\bm{x}_i) \right] \\
		&& \quad -\frac{1}{2} \left[ \sum_{i=1}^k  \left(\nabla h(\bm{x}_i;{\bm{\gamma}}_0)^T \frac{\bm{\lambda}}{\sqrt{n}}+ \frac{1}{2n}\bm{\lambda}^T h^{''}\!(\bm{x}_i;\bm{\bar\gamma})
		\bm{\lambda} \right)\!\xi(\bm{x}_i) \right]^2 \\
		&& \quad - \sum_{i=1}^k  \left[  \left(\nabla h(\bm{x}_i;{\bm{\gamma}}_0)^T \frac{\bm{\lambda}}{\sqrt{n}}+ \frac{1}{2n}\bm{\lambda}^T h^{''}\!(\bm{x}_i;\bm{\bar\gamma})
		\bm{\lambda} \right)\right]\!\xi(\bm{x}_i) \\
		&& + \frac{1}{2}  \sum_{i=1}^k  \left[ \left(\nabla h(\bm{x}_i;{\bm{\gamma}}_0)^T \frac{\bm{\lambda}}{\sqrt{n}}+ \frac{1}{2n}\bm{\lambda}^T h^{''}\!(\bm{x}_i;\bm{\bar\gamma})
		\bm{\lambda} \right)\right]^2 \!\xi(\bm{x}_i) \\
		&& + \sum_{m=3}^{\infty} \, \frac{(-1)^{m+1}}{m} [(\E_\xi Z_n)^m-\E_\xi Z_n^m]
		\end{eqnarray*}	

The first terms of the two expansions simplify, and after some calculations, we obtain, as in Theorem \ref{th:3}:
$$
		I_{12}(\xi;\bm{\gamma}_1) =
		1 + \frac{\zeta(\xi,\bm{\lambda},\bm{\gamma}_0)}{n} + O(n^{-3/2}),
		$$
where the $O(n^{-3/2})$ term is now computed exactly:		
\begin{eqnarray*}O(n^{-3/2})	\!\!	&=& \!
		\frac{1}{2} \left[\sum_i \frac{1}{4n^2}\left(\bm{\lambda}^T h^{''}\!(\bm{x}_i;\bm{\bar\gamma})
		\bm{\lambda} \right)^2 \xi(\bm{x}_i)  - \frac{1}{4n^2} \left(\sum_i \bm{\lambda}^T h^{''}\!(\bm{x}_i;\bm{\bar\gamma})
		\bm{\lambda}  \xi(\bm{x}_i)\right)^2 \right]\\
		&& \quad -\frac{1}{2n} \left[\left(\sum_i \nabla h(\bm{x}_i;{\bm{\gamma}}_0)^T \frac{\bm{\lambda}}{\sqrt{n}}\xi(\bm{x}_i)\right) \left( \sum_i\bm{\lambda}^T h^{''}\!(\bm{x}_i;		  \bm{\bar\gamma})\bm{\lambda} \xi(\bm{x}_i)\right) + \right. \\
		&& \qquad \quad \left.- \left( \sum_i \nabla h(\bm{x}_i;{\bm{\gamma}}_0)^T \frac{\bm{\lambda}}{\sqrt{n}}  \bm{\lambda}^T h^{''}\!(\bm{x}_i;\bm{\bar\gamma})
		\bm{\lambda} \xi(\bm{x}_i) \right) \right]\\
		&& \quad + \sum_{m=3}^{\infty} \frac{(-1)^{m+1}}{m}\, [(\E_\xi Z_n)^m-\E_\xi Z_n^m]
		\end{eqnarray*}
		
Let $\operatorname{Var}_\xi f(\bm{X})$ and $\operatorname{Cov}_\xi (f(\bm{X}),g(\bm{X}))$ denote variances and covariances with respect to the design $\xi$, respectively. We can then write:
\begin{eqnarray*}
&&  n [ I_{12}(\xi;\bm{\gamma}_1)-1] - \zeta(\xi,\bm{\lambda},\bm{\gamma}_0)  =\\
&&\quad\qquad = \frac{1}{8n} \operatorname{Var}_{\xi} \left(\bm{\lambda}^T h^{''}\!(\bm{X};\bm{\bar\gamma})\bm{\lambda}\right) +\frac{1}{2\sqrt{n}} \operatorname{Cov}_\xi\left(\bm{\lambda}^T h^{''}\!(\bm{X};\bm{\bar\gamma})\bm{\lambda}, \nabla h(\bm{X};{\bm{\gamma}}_0)^T\bm{\lambda} \right) \\
&& \qquad\qquad \qquad + n  \sum_{m=3}^{\infty} \frac{(-1)^{m+1}}{m} \left[(\E_\xi Z_n)^m-\E_\xi Z_n^m\right]
\end{eqnarray*}
		
Thus, using the Cauchy-Shwartz bound $\operatorname{Cov}_\xi (f,g) \leq \sqrt{\operatorname{Var}_\xi(f) \operatorname{Var}_\xi (g)}$,  
\begin{eqnarray}\label{eq:unif_cont}
&& \left| n[ I_{12}(\xi;\bm{\gamma}_1)-1] - \zeta(\xi,\bm{\lambda},\bm{\gamma}_0) \right| \leq  
\frac{1}{8n} \operatorname{Var}_{\xi} \left(\bm{\lambda}^T h^{''}\!(\bm{X};\bm{\bar\gamma})\bm{\lambda}\right) \nonumber \\
&& \qquad\qquad\qquad \qquad + 
\frac{1}{2\sqrt{n}} \sqrt{\operatorname{Var}_\xi\left(\bm{\lambda}^T h^{''}\!(\bm{X};\bm{\bar\gamma})\bm{\lambda}\right) \operatorname{Var}_\xi\left(\nabla h(\bm{X};{\bm{\gamma}}_0)^T\bm{\lambda} \right)} \nonumber \\
&& \qquad\qquad\qquad \qquad +  \sum_{m=3}^{\infty}\frac{n}{m} \left| (\E_\xi Z_n)^m-\E_\xi Z_n^m \right|.
\end{eqnarray}

We need to prove that the right-hand side of (\ref{eq:unif_cont}) is bounded from above by a quantity that is inversely proportional to $n$ and independent of $\xi$.

Let us start with 
\begin{equation}
\sum_{m=3}^{\infty}\frac{n}{m} \,\left| (\E_\xi Z_n)^m-\E_\xi Z_n^m \right|\leq \sum_{m=3}^{\infty}\frac{n}{m} \left[ (\E_\xi |Z_n|)^m+\E_\xi |Z_n|^m\right]\leq \sum_{m=3}^{\infty}  \frac{2n}{m}\, \E_\xi |Z_n|^m,
\label{ineq:Zn}
\end{equation}
and taking into account the expression of $Z_n$ given in (\ref{Zn}), we have that
\begin{eqnarray}
\sum_{m=3}^{\infty} \frac{2n}{m} \, \E_\xi \left| Z_n\right|^m
&\leq&
\sum_{m=3}^{\infty} \frac{2n}{m} \,\E_\xi \left(
\left|\nabla h(\bm{X};{\bm{\gamma}}_0)^T \frac{\bm{\lambda}}{\sqrt{n}}\right|+ \left| \frac{1}{2n}\bm{\lambda}^T h^{''}\!(\bm{X};\bm{\bar\gamma})
\bm{\lambda}\right|\right)^m
\nonumber
\\
 &\leq& \!\!\!\!\sum_{m=3}^{\infty} \!\frac{2n}{m}\! \left[\! 2^m\, \E_\xi\!\left|\nabla h(\bm{x}_i;{\bm{\gamma}}_0)^T \frac{\bm{\lambda}}{\sqrt{n}}\right|^m \!\!\!\!+\! 
	2^m\,\E_\xi \!\left|\frac{1}{2n}\bm{\lambda}^T h^{''}\!(\bm{x}_i;\bm{\bar\gamma})\bm{\lambda}\right|^m 
	\right],\nonumber\\
	\label{ineq:3}
\end{eqnarray}
where the last inequality follows from the fact that,
 given $a, b\in \Real$ and $m \in \mathbb{N}$,
$$(|a|+|b|)^m = \sum_{j=0}^m {m\choose j} |a|^j |b|^{m-j} \leq 2^m \max\{|a|^m,|b|^m \}\leq 2^m(|a|^m+|b|^m),$$
with $a = \nabla h(\bm{X};{\bm{\gamma}}_0)^T \frac{\bm{\lambda}}{\sqrt{n}}$ and $b= \frac{1}{2n}\bm{\lambda}^T h^{''}\!(\bm{X};\bm{\bar\gamma})
		\bm{\lambda}$.
		
In addition, in  the right-hand side of (\ref{ineq:3})  $2/m<1$ as in the summation $m>2$, and thus:
\begin{eqnarray}
&& \sum_{m=3}^{\infty} \frac{2n}{m} \left[ 2^m\, \E_\xi\left|\nabla h(\bm{x}_i;{\bm{\gamma}}_0)^T \frac{\bm{\lambda}}{\sqrt{n}}\right|^m + 
2^m\,\E_\xi \left|\frac{1}{2n}\bm{\lambda}^T h^{''}\!(\bm{x}_i;\bm{\bar\gamma})\bm{\lambda}\right|^m 
 \right] \nonumber\\
&& \quad \leq \sum_{m=3}^{\infty} n \left[ \E_\xi\left|2 \nabla h(\bm{x}_i;{\bm{\gamma}}_0)^T \frac{\bm{\lambda}}{\sqrt{n}}\right|^m + 
\E_\xi \left|\frac{1}{n}\bm{\lambda}^T h^{''}\!(\bm{x}_i;\bm{\bar\gamma})\bm{\lambda}\right|^m 
 \right] \nonumber\\
 && \quad\leq 
 \sum_{m=3}^\infty  n \left[\left(\frac{2M_1   L}{\sqrt{n}} \right)^m+ \left(\frac{ M_2 L^2}{n}\right)^m\right],
 \label{ineq:4}
\end{eqnarray}
where the last inequality follows from the assumptions of the Theorem:
\begin{eqnarray*}
&& \E_\xi\left|2 \nabla h(\bm{x}_i;{\bm{\gamma}}_0)^T \frac{\bm{\lambda}}{\sqrt{n}}\right|^m
 \leq 
\sup_{\bm{x} \in \mathcal{X}} \left|2 \nabla h(\bm{x};{\bm{\gamma}}_0)^T \frac{\bm{\lambda}}{\sqrt{n}}\right|^m\\
&& \quad =
\left[\sup_{\bm{x} \in \mathcal{X}} \left|2 \nabla h(\bm{x};{\bm{\gamma}}_0)^T \frac{\bm{\lambda}}{\sqrt{n}}\right|\right]^m
\leq \left(\frac{2M_1  L}{\sqrt{n}} \right)^m 
\end{eqnarray*}
and 
$$ \E_\xi \!\left|\frac{1}{n}\bm{\lambda}^T h^{''}\!(\bm{x}_i;\bm{\bar\gamma})\bm{\lambda}\right|^m
\!\!\!\!\leq 
\sup_{\bm{x} \in \mathcal{X}} \left|\frac{1}{n}\bm{\lambda}^T h^{''}\!(\bm{x};\bm{\bar\gamma})\bm{\lambda}\right|^m
\!\!\!\!=\!\!
\left( \!\sup_{\bm{x} \in \mathcal{X}} \left|\frac{1}{n}\bm{\lambda}^T h^{''}\!(\bm{x};\bm{\bar\gamma})\bm{\lambda}\right| \right)^{\!\!\!m}
\!\!\! \leq \!\!
 \left(\!\frac{ M_2 L^2}{n} \!\right)^{\!\!\!m}\!\!.$$
 In addition,
if $n>n^*$ then (\ref{nstar}) is fulfilled and thus  $2 M_1 L/\sqrt{n}<1$ and $M_2 L^2/n<1$. As a consequence, the two power series in (\ref{ineq:4}) converge, and from (\ref{ineq:Zn})-(\ref{ineq:4}) we obtain:
\begin{eqnarray}
&& \sum_{m=3}^{\infty} \frac{n}{m} \left|(\E_\xi Z)^m-\E_\xi Z^m\right|  
\leq \sum_{m=3}^\infty  n \left[\left(\frac{2M_1 L}{\sqrt{n}} \right)^m+ \left(\frac{M_2 L^2}{n}\right)^m\right]\nonumber\\
&& \qquad = n\left[\left(1-\frac{2M_1   L}{\sqrt{n}}\right)^{-1} -1-\frac{2M_1   L}{\sqrt{n}}-\left(\frac{2M_1   L}{\sqrt{n}}\right)^2 +\right.\nonumber\\
&&\qquad \qquad  + \left.\left(1-\frac{M_2 L^2}{n}\right)^{-1}-1- \frac{ M_2 L^2}{n}-\left(\frac{ M_2 L^2}{n}\right)^2 \right]
\label{ineq:5}
\end{eqnarray}

In particular, from (\ref{nstar}) we have that $1-2M_1   L/\sqrt{n}>1/2$, and thus: 
$$\left(1-\frac{2M_1   L}{\sqrt{n}}\right)^{-1} -1-\frac{2M_1   L}{\sqrt{n}}-\left(\frac{2M_1   L}{\sqrt{n}}\right)^2 = 
\frac{\left(2M_1   L/\sqrt{n}\right)^3}{1-2M_1   L/\sqrt{n}}\leq \frac{16 M_1^3 L^3}{n^{3/2}}.
$$
Similarly, the upper bound in (\ref{nstar}) guarantees that:
$$\left(1-\frac{ M_2 L^2}{n}\right)^{-1}-1- \frac{M_2 L^2}{n}-\left(\frac{M_2 L^2}{n}\right)^2 \leq 
\frac{\left(M_2   L^2/n\right)^3}{1-M_2   L^2/n}\leq \frac{2 M_2^3 L^6}{n^{3}}$$

Using these two upper-bounds in (\ref{ineq:5}), we have:
\begin{equation}\label{eq:bound1}
\sum_{m=3}^{\infty} \frac{n}{m} \left|(\E_\xi Z)^m-\E_\xi Z^m\right|\leq  \frac{16 M_1^3 L^3}{n^{1/2}}+\frac{2 M_2^3 L^6}{n^{2}}.\end{equation}

It then remains to consider the other terms in the right-hand side of (\ref{eq:unif_cont}). 
Using the trivial upper bound $\operatorname{Var}(X)\leq \E X^2$ and arguments similar to those used above,
$$\operatorname{Var}_{\xi} \left(\bm{\lambda}^T h^{''}\!(\bm{X};\bm{\bar\gamma})\bm{\lambda}\right) \leq \E_\xi \left|\bm{\lambda}^T h^{''}\!(\bm{X};\bm{\bar\gamma})\bm{\lambda} \right|^2 \leq |\bm{\lambda}^T\bm{\lambda}|^2 \E_\xi\|h^{''}\!(\bm{X};\bm{\bar\gamma})\|_2^2 \leq L^4 M_2^2$$
and $$\operatorname{Var}_\xi \left(\nabla h(\bm{X};{\bm{\gamma}}_0)^T\bm{\lambda} \right) \leq \E_\xi \left(\nabla h(\bm{X};{\bm{\gamma}}_0)^T\bm{\lambda} \right)^2 \leq \E_\xi \|\nabla h(\bm{X};{\bm{\gamma}}_0)\|^2 \|\bm{\lambda}\|^2 \leq L^2 M_1^2. $$

Thus, putting everything together, we get:
$$ \left| n ( I_{12}(\xi;\bm{\gamma}_1)-1) - \zeta(\xi,\bm{\lambda},\bm{\gamma}_0) \right| \leq  \frac{16L^3 M_1^3}{\sqrt{n}} + \frac{2L^6 M_2^3}{n^2} + \frac{L^4 M_2^2}{8n}+ \frac{L^3 M_1 M_2}{2\sqrt{n}}$$
where the constants $K,M_1$ and $M_2$ are independent of $\xi$.

\end{proof}

\bibliographystyle{plainnat}
\bibliography{references}

\end{document}